\documentclass[12pt,amstex]{amsart}

\usepackage{epsfig}
\usepackage{amsmath}
\usepackage{amssymb}
\usepackage{amscd}
\usepackage{graphicx}
\usepackage{pstricks}
\usepackage{pst-node}

\newtheorem{thm}{Theorem}[section]
\newtheorem{lem}[thm]{Lemma}

\newtheorem{prop}[thm]{Proposition}
\newtheorem{ques}[thm]{Question}
\newtheorem{cor}[thm]{Corollary}
\newtheorem{ex}[thm]{Example}
\newtheorem{conj}[thm]{Conjecture}
\newtheorem{defn}[thm]{Definition}

\theoremstyle{definition}
\newtheorem{de}[thm]{Definition}
\theoremstyle{remark}
\newtheorem{rem}[thm]{Remark}

\numberwithin{equation}{section}

\input xy
\xyoption{all}

\def \N {\mathbb N}

\def \Z {\mathbb Z}
\def \R {\mathbb R}

\def \cZ {\mathcal Z}
\def \A {\mathcal A}
\def \F {\mathcal F}
\def \G {\mathcal{G}}
\def \U {\mathcal U}

\def \O {\mathcal{O}}

\def \T {\mathbb{T}}
\def \E {\mathbb{E}}

\def \Q {{\bf Q}}
\def \RP {{\bf RP}}

\def \id {{\rm id}}
\def \Ind {{\bf {\rm Ind}}}

\def \a {\alpha }

\def \ep {\epsilon}
\def \d {\delta}
\def \D {\Delta}

\def \W {\Omega}
\def\w {\omega}
\def \ov {\overline}
\def \lra{\longrightarrow}

\begin{document}
\title{infinite-step nilsystems, independence and complexity}

\author{Pandeng Dong}
\address{Wu Wen-Tsun Key Laboratory of Mathematics, USTC, Chinese Academy of Sciences and
Department of Mathematics, University of Science and Technology of China,
Hefei, Anhui, 230026, P.R. China.} \email{dopandn@mail.ustc.edu.cn}

\author{Sebasti\'an Donoso}
\address{Centro de Modelamiento Matem\'atico and Departamento de Ingenier\'{\i}a
Matem\'atica, Universidad de Chile, Av. Blanco Encalada 2120,
Santiago, Chile.} \email{sdonoso@dim.uchile.cl}

\author{Alejandro Maass}
\address{Centro de Modelamiento Matem\'atico and Departamento de Ingenier\'{\i}a
Matem\'atica, Universidad de Chile, Av. Blanco Encalada 2120,
Santiago, Chile.} \email{amaass@dim.uchile.cl}

\author{Song Shao}
\address{Department of Mathematics, University of Science and Technology of China,
Hefei, Anhui, 230026, P.R. China.} \email{songshao@ustc.edu.cn}

\author{Xiangdong Ye}\address{Department of Mathematics, University of Science and Technology of China,
Hefei, Anhui, 230026, P.R. China.}\email{yexd@ustc.edu.cn}

\thanks{P. Dong and X. Ye were
supported by NNSF of China (11071231). S. Donoso and A. Maass were
supported by Basal-CMM and FONDAP-CRG grants. S. Shao was supported
by NNSF of China (10871186) and Program for New Century Excellent
Talents in University}


\begin{abstract}
An $\infty$-step nilsystem  is an inverse limit of minimal
nilsystems. In this article is shown that a minimal distal system is
an $\infty$-step nilsystem if and only if it has no nontrivial pairs
with arbitrarily long finite IP-independence sets. Moreover, it is
proved that any minimal system without nontrivial pairs with
arbitrarily long finite IP-independence sets is an almost one to one
extension of its maximal $\infty$-step nilfactor, and each invariant ergodic
measure is isomorphic (in the measurable sense) to the Haar measure on some
$\infty$-step nilsystem. The question if such a system is uniquely
ergodic remains open. In addition, the topological complexity of an
$\infty$-step nilsystem is computed, showing that it is polynomial for each
nontrivial open cover.
\end{abstract}

\maketitle

\markboth{Infinite-step nilsystems, independence and
complexity}{Dong, Donoso, Maass, Shao, Ye}

\section{Introduction}
In this paper we introduce the notion of $\infty$-step nilsystem
and study its relationship with the concept of independence.
We also study its topological complexity. In this
section, first we discuss the motivations for this subject and then
we state the main results of the article.

\subsection{Motivations}

By a {\it topological dynamical system} (t.d.s. for short) we mean a
pair $(X,T)$, where $X$ is a compact metric space and $T:X\to X$ is
a homeomorphism.

There are several motivations for studying this subject. The first
one comes from the so called \emph{local entropy theory}, for a
survey see \cite{GY}. Each t.d.s. admits a maximal zero topological
entropy factor, and this factor is induced by the smallest closed
invariant equivalence relation containing \emph{entropy pairs}
\cite{BL}. In \cite{HY}, entropy pairs are characterized as those
pairs that admit an \emph{interpolating set} of positive density.
Later on, the notions of \emph{sequence} entropy pairs \cite{HLSY}
and \emph{untame} pairs (called scrambled pairs in \cite{H1}) were
introduced. In \cite{KL} the concept of \emph{independence} was
extensively studied and used to unify the afore mentioned notions.
Let $(X,T)$ be a t.d.s. and ${\A}=(A_1,\ldots,A_k)$ be a tuple of
subsets of $X$. We say that a subset $F\subseteq \Z_+$ is an {\it
independence set} for ${\A}$ if for any nonempty finite subset
$J\subseteq F$ and any $s=(s(j):j\in J) \in \{1,\ldots,k\}^J$ we
have $\bigcap_{j\in J}T^{-j}A_{s(j)}\not=\emptyset$. It is shown
that a pair of points $x,y$ in $X$ is a sequence entropy pair if and
only if each ${\A}=(A_{1},A_{2})$, with $A_{1}$ and $A_{2}$
neighborhoods of $x$ and $y$ respectively, has arbitrarily long
finite independence sets. Also, the pair is an untame pair if and
only if each ${\A}=(A_{1},A_{2})$ as before has infinite
independence sets. It is known that each t.d.s. admits a maximal
zero sequence entropy factor, i.e. a null factor \cite{HLSY}, which
is induced by the smallest closed invariant equivalence relation
containing sequence entropy pairs, and a maximal tame factor
\cite{KL}, which is induced by the smallest closed invariant
equivalence relation containing untame pairs. It was shown
(\cite{HLSY,KL,G07}) that a minimal null (resp. tame) system is an
almost 1-1 extension of an equicontinuous t.d.s. and is uniquely
ergodic. For a similar study see \cite{MS}. Moreover, in the
equicontinuous case the uniquely ergodic measure is measure
theoretical isomorphic to the Haar measure of the underlying Abelian
group.
\medskip

To get a better understanding of the role of the notion of
independence in t.d.s., in \cite{HLY, HLY1} the authors
systematically investigate the independence for a given collection
of subsets of $\Z_+$. For a finite subset $\{p_1,\ldots,p_m\}$ of
$\N$, the {\em finite IP-set} generated by $\{p_1,\ldots,p_m\}$ is
the set $\{\ep_1 p_1+\ldots+\ep_m p_m: \ep_i\in\{0,1\}, 1\le i\le
m\}\setminus \{0\}$. The notion of ${\rm Ind}_{fip}$-pair is
introduced and studied in \cite{HLY1}: a pair  of points $(x,y)$ in
$X$ is an ${\rm Ind}_{fip}$-pair if and only if each
${\A}=(A_{1},A_{2})$, with $A_{1}$ and $A_{2}$ neighborhoods of $x$
and $y$ respectively, has arbitrarily long finite IP-independence
sets. Among other results it is shown that the ${\rm
Ind}_{fip}$-pair relation has the lifting property, i.e. if
$\pi:(X,T)\lra (Y,S)$ is a factor map between two t.d.s. then
$\pi\times \pi({\rm Ind}_{fip}(X,T)) ={\rm Ind}_{fip}(Y,S)$, where
${\rm Ind}_{fip}(X,T)$ is the set of all ${\rm Ind}_{fip}$-pairs of
$(X,T)$. It is clear that,

\bigskip

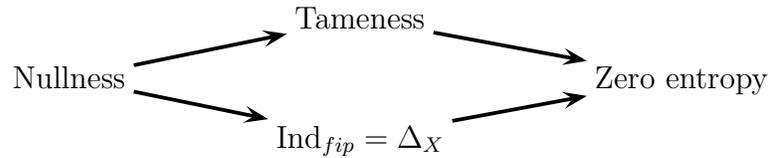
\begin{center}
$
 \begin{psmatrix}[colsep=2cm,rowsep=0.3cm]
 & &\text{Tameness}& \\
 & \text{Nullness} & & \text{Zero entropy} \\
 & &{\rm Ind}_{fip}=\Delta_X&
 \psset{arrows=->,nodesep=3pt}
 \everypsbox{\scriptstyle}
 \ncline[arrowscale=1.5,linewidth=0.05]{->}{2,2}{1,3}
 \ncline[arrowscale=1.5,linewidth=0.05]{->}{1,3}{2,4}
 \ncline[arrowscale=1.5,linewidth=0.05]{->}{2,2}{3,3}
 \ncline[arrowscale=1.5,linewidth=0.05]{->}{3,3}{2,4}
 \end{psmatrix}
$
\end{center}

\bigskip

\noindent So it is interesting to understand the dynamical properties of a minimal
t.d.s. without ${\rm Ind}_{fip}$-pairs.

\medskip

A second motivation comes from the study of the dynamics of nilsystems.
Let $(X,\mathcal{B}, \mu,T)$ be an ergodic system.

In \cite{HK05}, to study the convergence of some non-conventional
ergodic averages in this system, the authors proved that the
characteristic factors for such averages in $L^2(X,\mathcal{B},\mu)$
are $d$-step nilsystems for some integer $d\geq 1$ (see also
\cite{Z}). Then, in the topological setting,  in \cite{HKM} the
authors defined the notion of regionally proximal relation of order
$d$ associated to a t.d.s. $(X,T)$, $\RP^{[d]}$, and showed that if
the system is minimal and distal then $\RP^{[d]}$ is an equivalence
relation and $(X/\RP^{[d]},T)$ is the maximal $d$-step nilfactor of
the system. In a recent preprint \cite{SY}, this result was
generalized to arbitrarily minimal t.d.s. When studying minimal
distal systems carefully one finds that if $(x,y)\in \RP^{[d]}$ for
some integer $d$ then each ${\A}=(A_{1},A_{2})$, with $A_{1}$ and
$A_{2}$ neighborhoods of $x$ and $y$ respectively, has a finite
IP-independence set of length $n(d)$ such that $\lim_{d\to
\infty}n(d)=\infty$. This means that if
$(x,y)\in\RP^{[\infty]}=\cap_{d\ge 1} \RP^{[d]}$, then $(x,y)\in
{\rm Ind}_{fip}(X,T)$. This is the main reason leading us to define
$\infty$-step nilsystems and to study properties of minimal t.d.s.
without ${\rm Ind}_{fip}$-pairs.

\medskip

The third motivation comes from the theory of \emph{local
complexity} in topological dynamics. In \cite{BHM} the authors
introduced the notion of topological complexity for a t.d.s. using
open covers, and showed that a t.d.s. is equicontinuous if and only
if each nontrivial open cover has a bounded complexity. For a
further development, see \cite{HY1}. An interesting, but more
difficult question, is to understand t.d.s. with polynomial
complexity or extensions of such systems. It appears (this is proved
in Section \ref{section-complexity}) that inverse limits of
nilsystems are special systems with polynomial complexity.

\subsection{Main results of the paper}

In this paper, we study $\infty$-step nilsystems, which are t.d.s.
with trivial $\RP^{[\infty]}$, i.e.
$\RP^{[\infty]}=\cap_{d=1}^\infty\RP^{[d]}=\Delta$, the diagonal.
First, we prove that a minimal system is an $\infty$-step nilsystem
if and only if it is an inverse limit of minimal nilsystems. Then we
study the relation of $\infty$-step nilsystems and independence
pairs. It is proved that any minimal system without nontrivial ${\rm
Ind}_{fip}$-pairs is an almost one-to-one extension of its maximal
$\infty$-step nilfactor. Moreover, a minimal distal system is an
$\infty$-step nilsystem if and only if it has no nontrivial ${\rm
Ind}_{fip}$-pairs. We observe that there are plenty of minimal
distal systems which are not $\infty$-step nilsystems, though it is
not easy to construct explicit examples.

In addition, we show for any minimal system without nontrivial ${\rm
Ind}_{fip}$-pairs that each invariant ergodic measure is measure
theoretical isomorphic to the Haar measure on some $\infty$-step
nilsystem. We conjecture that such class of systems are uniquely
ergodic.

Finally, we prove the topological complexity of an $\infty$-step
nilsystem is polynomial for each non-trivial open cover.

\subsection{Organization of the paper}

We organize the paper as follows. In Section \ref{section-pre} we
introduce basic notions and facts we will meet in this article. Then
we define $\infty$-step nilsystems in Section
\ref{section-nilsystem}. In Section \ref{section-ind} we study the
relationship between $\infty$-step nilsystems and independence
pairs; and in Section \ref{section-example}, we give some examples.
In section \ref{uniqueergo} we study the conjecture concerning
unique ergodicity, and state some further ideas and questions.
Finally, in Section \ref{section-complexity}, we discuss the
complexity of an $\infty$-step nilsystem. Moreover, we give proofs
of some results stated in Section \ref{section-nilsystem} in the
Appendix.

\bigskip

\noindent {\bf Acknowledgements:} We thank W. Huang, H.F. Li and B.
Kra for useful comments and suggestions.


\section{Preliminaries}\label{section-pre}

\subsection{Topological dynamical systems}

A {\em transformation} of a compact metric space X is a
homeomorphism of X to itself. A {\em topological dynamical system}
(t.d.s.) or just a {\em system}, is a pair $(X,T)$, where $X$ is a
compact metric space and $T : X \rightarrow  X$ is a transformation.
We use $\rho (\cdot, \cdot)$ to denote the metric in $X$. In the
sequel, and if there is no confusion, in any t.d.s. we will always
use $T$ to indicate the transformation.

We will also make use of a more general
definition of a system. That is, instead
of just considering a single transformation $T$, we will consider commuting
homeomorphisms $T_1,\ldots , T_k$ of $X$. We recall some basic
definitions and properties of systems in the classical setting of
one transformation. Extensions to the general case are
straightforward.

\medskip

A system $(X, T)$ is {\em transitive} if there exists
$x\in X$ whose orbit $\O(x,T)=\{T^nx: n\in \Z\}$ is dense in $X$ and
such point is called a {\em transitive point}. The system is {\em
minimal} if the orbit of any point is dense in $X$. This property is
equivalent to saying that X and the empty set are the unique closed
invariant subsets of $X$.

\medskip

Let $(X,T)$ be a system and $\mathcal{M}(X)$ be the set of Borel
probability measures in $X$. A measure $\mu\in\mathcal{M}(X)$ is
{\em $T$-invariant} if for any Borel set $B$ of $X$,
$\mu(T^{-1}B)=\mu(B)$. Denote by $\mathcal{M}(X,T)$ the set of
invariant probability measures. A measure $\mu\in\mathcal{M}(X,T)$
is {\em ergodic} if for any Borel set $B$ of $X$ satisfying
$\mu(T^{-1}B\triangle B)=0$ we have $\mu(B)=0$ or $\mu(B)=1$. Denote
by $\mathcal{M}^e(X,T)$ the set of ergodic measures. The system
$(X,T)$ is  {\em uniquely ergodic} if $\mathcal{M}(X,T)$ consists of
only one element.

\medskip

A {\em homomorphism} between the t.d.s. $(X,T)$ and $(Y,T)$ is a
continuous onto map $\pi: X\rightarrow Y$ which intertwines the
actions; one says that $(Y,T)$ is a {\it factor} of $(X,T)$ and that
$(X,T)$ is an {\em extension} of $(Y,T)$. One also refers to $\pi$
as a {\em factor map} or an {\em extension} and one uses the
notation $\pi: (X,T) \rightarrow (Y,T)$. The systems are said to be
{\em conjugate} if $\pi$ is a bijection. An extension $\pi$ is
determined by the corresponding closed invariant equivalence
relation $R_{\pi} = \{ (x_{1},x_2): \pi (x_1)= \pi (x_2) \} =(\pi
\times \pi )^{-1} \Delta_Y \subset  X \times X$, where $\Delta_Y$ is
the diagonal on $Y$. An extension $\pi: (X,T) \rightarrow (Y,T)$ is
{\em almost one-to-one} if the $G_\delta$ set $X_0=\{x \in X:
\pi^{-1}(\pi(x)) = \{x\}\}$ is dense.

\subsection{Distality and Proximality}

Let $(X,T)$ be a t.d.s. A pair $(x,y)\in X\times X$ is a {\it proximal}
pair if
\begin{equation*}
\inf_{n\in \Z} \rho (T^nx,T^ny)=0
\end{equation*}
and is a {\em distal} pair if it is not proximal. Denote by
$P(X,T)$ or $P_X$ the set of proximal pairs of $(X,T)$. The t.d.s.
$(X,T)$ is {\em distal} if $(x,y)$ is a distal pair whenever
$x,y\in X$ are distinct.

\medskip

An extension $\pi : (X,T) \rightarrow (Y,S)$ is  {\em
proximal} if $R_{\pi} \subset P(X,T)$ and is {\em distal} if
$R_\pi\cap P(X,T)=\Delta_X$. Observe that when $Y$ is trivial (reduced to one point) the map
$\pi$ is distal if and only if $(X,T)$ is distal.

\subsection{Independence}

The notion of \emph{independence} was firstly introduced and studied in
\cite[Definition 2.1]{KL}. It corresponds to a modification of the notion of
\emph{interpolating} studied in \cite{GW, HY} and was considerably discussed in
\cite{HLY, HLY1}.

\begin{de}
Let $(X, T)$ be a t.d.s. Given a tuple $\A=(A_1,\ldots,A_k)$ of
subsets of $X$ we say that a subset $F\subset\Z_{+}$ is an {\em
independence set} for $\A$ if for any nonempty finite subset
$J\subset F$ and any $s=(s(j): j\in J ) \in\{1,\ldots,k\}^J$ we have
$$\bigcap\limits_{j\in J}T^{-j}A_{s(j)}\neq\emptyset \ .$$

We shall denote the collection of all independence sets for $\A$ by
$\Ind(A_1,\ldots,A_k)$ or $\Ind\A$.

\end{de}

A finite subset $F$ of $\Z_{+}$ is called a {\em finite IP-set}  if
there exists a finite subset $\{p_1,p_2,\ldots,p_m\}$ of $\N$ such
that
$$F=FS(\{p_i\}_{i=1}^m)=\{p_{i_1}+\cdots+p_{i_k}:1\leq i_1<\cdots<i_k\leq m\}.$$
Now we define ${\rm Ind}_{fip}$-pairs.

\begin{de}
Let $(X, T)$ be a t.d.s. A pair $(x_1,x_2) \in X \times X$ is called
an {\em ${\rm Ind}_{fip}$-pair} if for any neighborhoods $U_1$,
$U_2$ of $x_1$ and $x_2$ respectively, $\Ind(U_1,U_2)$ contains
arbitrarily long finite IP-sets. Denote by ${\rm Ind}_{fip}(X,T)$
the set of all ${\rm Ind}_{fip}$-pairs of $(X,T)$.
\end{de}

\subsection{Parallelepipeds}

Let $X$ be a set, $d\ge 1$ be an integer, and write $[d] = \{1,
2,\ldots , d\}$. We view $\{0, 1\}^d$ in one of two ways, either as
a sequence $\ep=\ep_1\ldots \ep_d :=(\ep_1,\ldots,\ep_d)$ of $0'$s
and $1'$s; or as a subset of $[d]$. A subset $\ep$ corresponds to
the sequence $(\ep_1,\ldots, \ep_d)\in \{0,1\}^d$ such that $i\in
\ep$ if and only if $\ep_i = 1$ for $i\in [d]$. For example, ${\bf
0}=(0,0,\ldots,0)\in \{0,1\}^d$ is the same as $\emptyset \subset
[d]$.

If ${\bf n} = (n_1,\ldots, n_d)\in \Z^d$ and $\ep\in \{0,1\}^d$, we
define
$${\bf n}\cdot \ep = \sum_{i=1}^d n_i\ep_i .$$
If we consider $\ep$ as $\ep\subset [d]$, then ${\bf n}\cdot \ep  =
\sum_{i\in \ep} n_i .$

\medskip

We denote $X^{2^d}$ by $X^{[d]}$. A point ${\bf x}\in X^{[d]}$ can
be written in one of two equivalent ways, depending on the context:
$${\bf x} = (x_\ep :\ep\in \{0,1\}^d )= (x_\ep : \ep\subset [d]). $$
Hence $x_\emptyset =x_{\bf 0}$ is the first coordinate of ${\bf x}$.
As examples, points in $X^{[2]}$ are written
$$(x_{00},x_{10},x_{01},x_{11})=(x_{\emptyset}, x_{\{1\}},x_{\{2\}},x_{\{1,2\}}),$$
and points in $X^{[3]}$ look like
\begin{eqnarray*}
 & & (x_{000},x_{100},x_{010},x_{110},x_{001},x_{101},x_{011},x_{111}) \\
     &=&(x_{\emptyset}, x_{\{1\}},x_{\{2\}},x_{\{1,2\}},x_{\{3\}}, x_{\{1,3\}},x_{\{2,3\}},x_{\{1,2,3\}}).
\end{eqnarray*}

\medskip

For $x \in X$ we write $x^{[d]} = (x, x,\ldots , x)\in  X^{[d]}$.
The diagonal of $X^{[d]}$ is $\D^{[d]} = \{x^{[d]}: x\in X\}$.
Usually, when $d=1$, we denote the diagonal by $\D_X$ or $\D$ instead of
$\D^{[1]}$.

A point ${\bf x} \in X^{[d]}$ can be decomposed as ${\bf x} = ({\bf
x'},{\bf  x''})$ with ${\bf x}', {\bf x}''\in X^{[d-1]}$, where
${\bf x}' = (x_{\ep0} : \ep\in \{0,1\}^{d-1})$ and ${\bf x}''=
(x_{\ep1} : \ep\in \{0,1\}^{d-1})$. We can also isolate the first
coordinate, writing $X^{[d]}_* = X^{2^d-1}$ and then writing a point
${\bf x}\in X^{[d]}$ as ${\bf x} = (x_\emptyset, {\bf x}_*)$, where
${\bf x}_*= (x_\ep : \ep\neq \emptyset) \in X^{[d]}_*$.

\medskip

Identifying $\{0, 1\}^d$ with the set of vertices of the Euclidean
unit cube, a Euclidean isometry of the unit cube permutes the
vertices of the cube and thus the coordinates of a point ${\bf x} \in
X^{[d]}$. These permutations are the {\em Euclidean permutations} of
$X^{[d]}$.

\subsection{Dynamical parallelepipeds}
We follow definitions from \cite{HKM}.

Let $(X, T)$ be a t.d.s. and $d\ge 1$ be
an integer.

We define the set of (dynamical) {\em parallelepipeds of
dimension $d$}, $\Q^{[d]}(X)$ (or just $\Q^{[d]}$), as the closure in $X^{[d]}$
of elements of the form
$$(T^{{\bf n}\cdot \ep}x=T^{n_1\ep_1+\ldots
+ n_d\ep_d}x: \ep= (\ep_1,\ldots,\ep_d)\in\{0,1\}^d) ,$$
where ${\bf n} = (n_1,\ldots , n_d)\in \Z^d$ and $ x\in X$.
It is important to note that $\Q^{[d]}$ is invariant under the
Euclidean permutations of $X^{[d]}$.

As examples, $\Q^{[2]}$ is the closure in $X^{[2]}=X^4$ of the set
$$\{(x, T^mx, T^nx, T^{n+m}x) : x \in X, m, n\in\Z\}$$ and $\Q^{[3]}$
is the closure in $X^{[3]}=X^8$ of the set $$\{(x, T^mx, T^nx,
T^{m+n}x, T^px, T^{m+p}x, T^{n+p}x, T^{m+n+p}x) : x\in X, m, n, p\in
\Z\}.$$

Let $\phi: (X,T) \rightarrow (Y,T)$ be a factor map.
Define$\phi^{[d]}: X^{[d]}\rightarrow Y^{[d]}$ by $(\phi^{[d]}{\bf
x})_\ep=\phi x_\ep$ for every ${\bf x}\in X^{[d]}$ and every
$\ep\subset [d]$. The {\em diagonal transformation} of $X^{[d]}$ is
the map $T^{[d]}$. We define {\em face transformations} inductively
as follows: Let $T^{[0]}=T$, $T^{[1]}_1=\id \times T$. If
$\{T^{[d-1]}_j\}_{j=1}^{d-1}$ is already defined, then set
$$T^{[d]}_j=T^{[d-1]}_j\times T^{[d-1]}_j, \ j\in \{1,2,\ldots, d-1\},$$
$$T^{[d]}_d=\id ^{[d-1]}\times T^{[d-1]}.$$

It is easy to see that for $j\in [d]$, the face transformation
$T^{[d]}_j : X^{[d]}\rightarrow X^{[d]}$ can be defined, for
every ${\bf x} \in X^{[d]}$ and $\ep\subset [d]$, by

$$T^{[d]}_j{\bf x}=
\left\{
  \begin{array}{ll}
    (T^{[d]}_j{\bf x})_\ep=Tx_\ep, & \hbox{$j\in \ep$;} \\
    (T^{[d]}_j{\bf x})_\ep=x_\ep, & \hbox{$j\not \in \ep$.}
  \end{array}
\right.$$

\medskip

The {\em face group} of dimension $d$ is the group $\F^{[d]}(X)$ of
transformations of $X^{[d]}$ spanned by the face transformations.
The {\em parallelepiped group} of dimension $d$ is the group
$\G^{[d]}(X)$ spanned by the diagonal transformation and the face
transformations. We often write $\F^{[d]}$ and $\G^{[d]}$ instead of
$\F^{[d]}(X)$ and $\G^{[d]}(X)$, respectively. For $\G^{[d]}$ and
$\F^{[d]}$, we use similar notations to that used for $X^{[d]}$:
Namely, an element of either of these groups is written as $S =
(S_\ep : \ep\in\{0,1\}^d)$. In particular, $\F^{[d]} =\{S\in
\G^{[d]}: S_\emptyset ={\rm id}\}$.

\medskip

For convenience, we denote the orbit closure of ${\bf x}\in X^{[d]}$
under $\F^{[d]}$ by $\overline{\F^{[d]}}({\bf x})$, instead of
$\overline{\O({\bf x}, \F^{[d]})}$.

It is easy to verify that $\Q^{[d]}$ is the closure in $X^{[d]}$ of
$$\{Sx^{[d]} : S\in \F^{[d]}, x\in X\}.$$ If $(X,T)$ is a
transitive system and $x$ a transitive point, then $\Q^{[d]}$ is the
closed orbit of $x^{[d]}$ under the group $\G^{[d]}$. Moreover, it
is easy to get that
$$\Q^{[d]}=\ov{\{(T^{{\bf n}\cdot\ep}x)_{\ep\in\{0,1\}^d}
:{\bf n}\in\N^d, x\in X\}} \ .$$

\subsection{Nilmanifolds and nilsystems}

Let $G$ be a group. For $g, h\in G$ and $A,B \subset G$, we write $[g, h] =
ghg^{-1}h^{-1}$ for the commutator of $g$ and $h$ and
$[A,B]$ for the subgroup spanned by $\{[a, b] : a \in A, b\in B\}$.
The commutator subgroups $G_j$, $j\ge 1$, are defined inductively by
setting $G_1 = G$ and $G_{j+1} = [G_j ,G]$. Let $d \ge 1$ be an
integer. We say that $G$ is {\em $d$-step nilpotent} if $G_{d+1}$ is
the trivial subgroup.

\medskip

Let $G$ be a $d$-step nilpotent Lie group and $\Gamma$ be a discrete
cocompact subgroup of $G$. The compact manifold $X = G/\Gamma$ is
called a {\em $d$-step nilmanifold}. The group $G$ acts on $X$ by
left translations and we write this action as $(g, x)\mapsto gx$.
The Haar measure $\mu$ of $X$ is the unique probability measure on
$X$ invariant under this action. Let $\tau\in G$ and $T$ be the
transformation $x\mapsto \tau x$ of $X$. Then $(X, \mu, T)$ is
called a {\em $d$-step nilsystem}. In the topological setting we omit the measure
and just say that $(X,T)$ is a $d$-step nilsystem.

\medskip

We will need to use inverse limits of nilsystems, so we recall the
definition of a sequential inverse limit of systems. If
$(X_i,T_i)_{i\in \N}$ are systems with $diam(X_i)\le 1$ and
$\pi_i: X_{i+1}\rightarrow X_i$ are factor maps, the {\em inverse
limit} of the systems is defined to be the compact subset of
$\prod_{i\in \N}X_i$ given by $\{ (x_i)_{i\in \N }: \pi_i(x_{i+1}) =
x_i\}$, and we denote it by
$\lim\limits_{\longleftarrow}(X_i,T_i)_{i\in\N}$. It is a compact
metric space endowed with the distance $\rho((x_{i})_{i\in\N}, (y_{i})_{i\in
\N}) = \sum_{i\in \N} 1/2^i \rho_i(x_i, y_i )$, where $\rho_{i}$ is the metric in
$X_{i}$. We note that the
maps $T_i$ induce naturally a transformation $T$ on the inverse
limit.

\medskip

The following structure theorem characterizes inverse limits of
nilsystems using dynamical parallelepipeds.

\begin{thm}[Host-Kra-Maass]\cite[Theorem 1.2.]{HKM}\label{HKM}
Assume that $(X, T)$ is a transitive topological dynamical system
and let $d \ge 2$ be an integer. The following properties are
equivalent:
\begin{enumerate}
  \item If ${\bf x}, {\bf y} \in \Q^{[d]}$ have $2^d-1$ coordinates in common, then ${\bf x} = {\bf y}$.
  \item If $x, y \in X$ are such that $(x, y,\ldots , y) \in  \Q^{[d]}$,
then $x = y$.
  \item $X$ is an inverse limit of $(d-1)$-step minimal
nilsystems.
\end{enumerate}
\end{thm}

A transitive system satisfying one of the equivalent properties
above is called  a {\em system of order $(d-1)$}.

\section{$\infty$-step nilsystems}\label{section-nilsystem}

\subsection{Regionally proximal relation
of order $d$}
First we recall a fundamental relation introduced in \cite{HKM} allowing to
characterize maximal nilfactors in \cite{HKM} (for minimal distal systems) and
in \cite{SY} (for general minimal systems).

\begin{de}
Let $(X, T)$ be a system and let $d\in \N$. The points $x, y \in X$ are
said to be {\em regionally proximal of order $d$} if for any $\d  >
0$, there exist $x', y'\in X$ and a vector ${\bf n} = (n_1,\ldots ,
n_d)\in\Z^d$ such that $\rho (x, x') < \d, \rho (y, y') <\d$, and $$
\rho (T^{{\bf n}\cdot \ep}x', T^{{\bf n}\cdot \ep}y') < \d\
\text{for any nonempty $\ep\subset [d]$}.$$ In other words, there
exists $S\in \F^{[d]}$ such that $\rho (S_\ep x', S_\ep y') <\d$ for
every $\ep\neq \emptyset$. The set of regionally proximal pairs of
order $d$ is denoted by $\RP^{[d]}$ (or by $\RP^{[d]}(X,T)$ in case of
ambiguity), and is called {\em the regionally proximal relation of
order $d$}.
\end{de}

It is easy to see that $\RP^{[d]}$ is a closed and invariant
relation. Observe that
\begin{equation*}
    {\bf P}(X,T)\subseteq  \ldots \subseteq \RP^{[d+1]}\subseteq
    \RP^{[d]}\subseteq \ldots \RP^{[2]}\subseteq \RP^{[1]}=\RP(X,T).
\end{equation*}
\medskip

The following theorems proved in \cite{HKM} (for minimal distal systems) and
in \cite{SY} (for general minimal systems) tell us conditions under which
$(x,y)$ belongs to $\RP^{[d]}$ and the relation between $\RP^{[d]}$ and
$d$-step nilsystems.

\begin{thm}\label{thm-1}
Let $(X, T)$ be a minimal system and let $d\in \N$. Then
\begin{enumerate}
\item $(x,y)\in \RP^{[d]}$ if and only if $(x,y,\ldots,y)\in \Q^{[d+1]}$
if and only if $(x,y,\ldots,y) \in
\overline{\F^{[d+1]}}(x^{[d+1]})$.

\item $\RP^{[d]}$ is an equivalence relation.

\item $(X,T)$ is a system of order $d$ if and only if $\RP^{[d]}=\Delta_X$.
\end{enumerate}
\end{thm}

\subsection{$\infty$-step Nilsystems}
The regionally proximal relation of order $d$ allows to construct the maximal $d$-step
nilfactor of a system. That is, any factor of order $d$ (inverse limit of $d$-step minimal nilsystems) factorize through this system.

\begin{thm}\label{thm0}\cite{SY}
Let $\pi: (X,T)\rightarrow (Y,T)$ be a factor map between minimal systems
and let $d\in \N$. Then,
\begin{enumerate}
  \item $\pi\times \pi (\RP^{[d]}(X,T))=\RP^{[d]}(Y,T)$.
  \item $(Y,T)$ is a system of order $d$ if and only if $\RP^{[d]}(X,T)\subset R_\pi$.
\end{enumerate}
In particular, the quotient of $(X,T)$ under $\RP^{[d]}(X,T)$ is the
maximal $d$-step nilfactor of $X$ (i.e. the maximal factor of order
$d$).
\end{thm}

It follows that for any minimal system $(X,T)$,
$$\RP^{[\infty]}=\bigcap\limits_{d\ge 1} \RP^{[d]}$$
is a closed invariant equivalence relation (we write $\RP^{[\infty]}(X,T)$ in case of ambiguity). Now we formulate the
definition of $\infty$-step nilsystems or systems of order $\infty$.

\begin{de}
A minimal system $(X, T)$ is an {\em $\infty$-step
nilsystem} or {\em a system of order $\infty$}, if the equivalence
relation $\RP^{[\infty]}$ is trivial, i.e. coincides with the
diagonal.
\end{de}

\begin{rem}
Similar to Theorem \ref{thm0}, one can show that the quotient of a
minimal system $(X,T)$ under $\RP^{[\infty]}$ is the maximal
$\infty$-step nilfactor of $(X,T)$.
\end{rem}

Let $(X,T)$ be a minimal system. It is easy to see that if $(X,T)$
is an inverse limit of minimal nilsystems, then $(X,T)$ is an
$\infty$-step nilsystem. Conversely, if $(X,T)$ is a minimal
$\infty$-step nilsystem, then $\RP^{[\infty]}=\Delta_X$. For any
integer $d\geq 1$ let $(X_d,T)$ be the quotient of $(X,T)$ under
$\RP^{[d]}$. Then $(X,T)=\displaystyle
\lim_{\longleftarrow}(X_d,T)_{d\in \N}$ as
$\Delta_X=\RP^{[\infty]}=\bigcap\limits_{d\geq 1}\RP^{[d]}$. In fact
we can show more as the following theorem says.

\begin{thm}
A minimal system is an $\infty$-step nilsystem if and only if it is
an inverse limit of minimal nilsystems.
\end{thm}
\begin{proof} It remains to prove that if $(X,T)$ is a minimal $\infty$-step
nilsystem, then it is an inverse limit of minimal nilsystems. First
we may assume that $(X,T)=\displaystyle
\lim_{\longleftarrow}(X_d,T)_{d\in \N}$, where $X_d=X/\RP^{[d]}$ for
any $d\ge 1$. By Theorem \ref{HKM}, for any $d\geq 1$, $(X_d,T)$ is
an inverse limit of minimal $d$-step nilsystems.

We need the following claim.

\medskip

\noindent {\bf Claim:} {\em Let $(Y,S)$ be a minimal system, and let
$(Y_i,S)$ be factors of $(Y,S)$ which are $k_i$-step nilsystems,
where $1\le i\le n$ and $\max\{k_i:1\le i\le n\}= k$. Then there
exists a $k$-step nilsystem $(Z,S)$ such that it is a factor of
$(Y,S)$ and is an extension of $(Y_i,S)$ for $1\leq i\leq n$.}

\medskip
\noindent {\em Proof of Claim:} Let $\pi_i$ be the factor map
between $(Y,S)$ and $(Y_i,S)$ and assume $(Y_i,S)$ has the form of
$(H_i/\Gamma_i, h_i)$, where $h_i\in H_i$ and $S$ is the
left translation by $h_i$ on $Y_i$. Set
$G=H_1\times\cdots\times H_n$, $\Gamma=\Gamma_1\times\cdots
\Gamma_n$ and $g=(h_1,\ldots,h_n)$. Then $G$ is a $k$-step nilpotent
Lie group and $\Gamma$ is a discrete uniform subgroup of $G$. Let
$S: G/\Gamma\rightarrow G/\Gamma$ be the transformation $x\mapsto g
x$. Choose any point $y\in Y$ and let $Z=\overline{\{g^n (\pi_1(y),
\ldots,\pi_n(y)):n\in\Z\}}\subset G/\Gamma$ be the orbit closure of
$(\pi_1(y), \ldots,\pi_n(y))$ under $S$. Since nilsystems are
distal, $(Z,S)$ is minimal. Moreover it is a $k$-step nilsystem
\cite{Le}. And of course, $(Z,S)$ is a factor of $(Y,S)$ and it is
an extension of $(Y_i,S)$ for $1\leq i\leq n$.

\medskip

Now we show that $(X,T)$ is an inverse limit of minimal nilsystems
using previous claim. As $(X_1,T)$ is a system of order 1, it is an
inverse limit of some 1-step nilsystems $((X_1)_i,T)_{i\in\N}$ by
Theorem \ref{HKM}. Similarly, as $(X_2,T)$ is a system of order 2,
it is an inverse limit of some 2-step nilsystems
$((X_2)_i,T)_{i\in\N}$. Note that all $((X_1)_i,T)_{i\in\N}$ and
$((X_2)_i,T)_{i\in\N}$ are factors of $X_2$. By the above claim, we
may reconstruct 2-step nilsystems $((X_2)_i,T)_{i\in\N}$ such that
for all $i\in\N$, $((X_2)_{i+1},T)$ is an extension of
$((X_1)_1,T)$, $((X_1)_2,T)$, $\ldots$, $((X_1)_{i+1},T)$ and
$((X_2)_1,T),\ldots,((X_2)_i,T)$.

Similarly and inductively, for any given $k\in\N$ $(X_k,T)$ can be
written as the inverse limit of some $k$-step nilsystems
$((X_k)_i,T)_{i\in\N}$ satisfying that for all $i\in\N$,
$((X_{k+1})_i,T)$ is an extension of $((X_k)_i,T)$.

Since $(X,T)$ is the inverse limit of $(X_k,T)_{k\in\N}$,
$\bigcup_{k\in\N}C(X_k)$ is dense in $C(X)$. And as $(X_k,T)$ is the
inverse limit of $((X_k)_i,T)_{i\in\N}$, we have that
$\bigcup_{k\in\N}\bigcup_{i\in\N} C((X_k)_i)$ is dense in $C(X)$. So
we may choose a sequence of $(k_n,i_n)_{n\in\N}\subset \N\times\N$
with $k_n<k_{n+1},i_n<i_{n+1}$ such that
$\bigcup_{n\in\N}C((X_{k_n})_{i_n})$ is dense in $C(X)$. Thus
$(X,T)$ is the inverse limit of $((X_{k_n})_{i_n},T)_{n\in\N}$. That
is, $(X,T)$ is an inverse limit of minimal nilsystems. This
completes the proof of the theorem.

\end{proof}

Since minimal nilsystems are uniquely ergodic, it is easy to see
that minimal $\infty$-step nilsystems are also uniquely ergodic.
\medskip

\subsection{About $\RP^{[\infty]}=\RP^{[d]}$}

Observe that if $\Q^{[d+1]}= \Q^{[d]}\times \Q^{[d]}$ then
$\RP^{[d-1]}=\RP^{[d]}$. Indeed, if $(x,y,\ldots,y)$ and
$(y,\ldots,y)\in\Q^{[d]}$, then
$(x,y,\ldots,y,y,\ldots,y)\in\Q^{[d+1]}$. Moreover, the system is
weakly-mixing and thus $\RP^{[d]}=X\times X$ for any $d\geq 1$,
as shows the following proposition.

\begin{prop}
Let $(X,T)$ be a minimal system.
If $\Q^{[d+1]}=\Q^{[d]}\times
\Q^{[d]}$ for some $d\in \N$ then $X$ is weakly-mixing and hence
$\Q^{[d]}=X^{[d]}$ for any $d \in \N$.
\end{prop}

\begin{proof}
Let $x,y,a\in X$. By minimality $(x,y,x,y,\ldots,x,y)
\in \Q^{[d]}$ and by hypothesis the point
${\bf{x}}=(x,y,x,y,\ldots,x,y,a,\ldots ,a) \in \Q^{[d+1]}$.
If $d=1$, $(x,y,a,a)\in \Q^{[2]}$ and then $(x,y)\in \RP^{[1]}$.
For any integer $d>1$, applying Euclidean permutations, we get that
${\bf{y}}=(x,y,\ldots,x,y,a,\ldots,a,x,y\dots,x,y,a,\ldots,a ) \in \Q^{[d+1]}$ too.
Considering the first half of ${\bf{y}}$ and iterating the process we
finish in the case $d=1$. We conclude $\RP^{[1]}=X\times X$ and the result
follows.
\end{proof}

Now, if $\RP^{[d-1]}=\RP^{[d]}$ for some $d$ the following theorem
states that all regionally proximal relations of higher order
coincide and thus $\RP^{[\infty]}=\RP^{[d-1]}$. This result is
natural but its proof is somewhat involved, so we leave it for the
appendix. For the definition of $Z_d$ see Section 7.

\begin{thm}\label{noteasy}
\begin{enumerate}

\item Let $(X,T)$ be a minimal system. If $\RP^{[d]}=\RP^{[d+1]}$ for
some $d\in\N$, then $\RP^{[n]}=\RP^{[d]}$ for all $n\ge d$.

\item Let $(X,\mathcal{B},\mu,T)$ be an ergodic measure preserving
system. If $Z_d=Z_{d+1}$ for some $d\in \N$, then $Z_n=Z_d$ for each
$n\ge d$.

\item Let $(X,T)$ be a minimal system and $\mu$ be an ergodic Borel
probability measure on $X$. If $Z_n$ is isomorphic (with respect to
the corresponding invariant measure) to $X_{n}=X/\RP^{[n]}$ for some
$n\in \N$, then $Z_k$ is isomorphic to $X_{k}=X/\RP^{[k]}$ for all
$k\le n$.
\end{enumerate}
\end{thm}

\section{The structure of minimal systems without nontrivial
${\rm Ind}_{fip}$-pairs} \label{section-ind}

In this section we discuss the structure of minimal systems without
nontrivial ${\rm Ind}_{fip}$-pairs. We will show that such systems
are almost one-to-one extensions of their maximal $\infty$-step nilfactors.

\subsection{A criterion to be an ${\rm Ind}_{fip}$-pair}

First we characterize ${\rm Ind}_{fip}$-pairs using dynamical parallelepipeds.

\medskip

Let $(X,T)$ be a transitive system. It is easy to check that ${\bf
x}=(x_\ep:\ep\subset[d])\in\Q^{[d]}$ if and only if for any
neighborhood $U_\ep$ of $x_\ep$ respectively, there exist positive
integers ${\bf n}=(n_1,n_2,\ldots,n_d)\in\N^d$ such that
$\bigcap_{\ep\subset[d]}T^{-{\bf n}\cdot\ep}U_\ep\neq\emptyset$.
Moreover, the point in $\bigcap_{\ep\subset[d]}T^{-{\bf
n}\cdot\ep}U_\ep$ can be chosen to be in the orbit of a transitive
point.

\begin{lem}\label{lem}
Let $(X,T)$ be a transitive t.d.s. and $(x_1, x_2)\in X\times X$
with $x_1\not=x_2$. Then, $\{x_1,x_2\}^{[d]}\subset \Q^{[d]}$ for
all integer $d\geq 1$ if and only if $(x_1,x_2)$ is an ${\rm
Ind}_{fip}$-pair.
\end{lem}

\begin{proof}
Let $(x_1, x_2)\in X\times X$ with $x_1\not=x_2$.
First, we assume $\{x_1,x_2\}^{[d]}\subset \Q^{[d]}$ for all integer $d\geq 1$.

Let $U_1$, $U_2$ be neighborhoods of $x_1$ and $x_2$ respectively and fix $d \in \N$.
We show there exist positive integers $\{p_1,p_2,\ldots,p_d\}$ such that
$$F=FS(\{p_i\}_{i=1}^d)=\{p_{i_1}+\cdots+p_{i_k}:1\leq i_1<\cdots<i_k\leq d\}$$
belongs to $\Ind(U_1,U_2)$.

Since $\{x_1,x_2\}^{[d]}$ has $2^{2^d}$ elements we write it
$\{{\bf x_\eta}: \eta\subset [2^d]\}$.
Now, for any $\ep\subset [d]$
and $\eta\subset [2^d]$, let $x_{\ep\eta}=(\bf x_\eta)_\ep$ and construct
the point ${\bf x}=(x_{\ep\eta}:\ep\subset [d], \eta\subset
[2^d])\in \{x_1,x_2\}^{[2^d+d]}$. Clearly, if we choose an
identification of coordinates, we can write
${\bf x}=(x_{\rho}:\rho\subset [2^d+d])$.
For any $\rho\subset [2^d+d]$, let
$U_\rho=U_i$ if $x_\rho=x_i$, $i=1,2$. Then the product set
$\bigotimes_{\rho\subset [2^d+d]}U_\rho$ is a
neighborhood of ${\bf x}$.
Since, by hypothesis, ${\bf x}\in\Q^{[2^d+d]}$, then
there exist $x\in X$, $p_0\in \N$ and ${\bf
p}=\{p_1,\ldots,p_{2^d+d}\}\subset \N$ such that for any $\rho
\subset [2^d+d]$, $T^{p_0+{\bf p\cdot\rho}}x\in U_\rho$.

We show that the finite IP set $F$ generated by $\{p_1,p_2,\ldots,p_d\}$
belongs to $\Ind(U_1,U_2)$. For any $s \in
\{1,2\}^{2^d}=\{1,2\}^{[d]}$, since $\{x_{s(\ep)}:\ep\subset
[d]\}\in\{x_1,x_2\}^{[d]}$, there exists $\eta \subset [2^d]$
such that ${\bf x_\eta}=\{x_{s(\ep)}:\ep\subset [d]\}$, i.e. for any
$\ep\subset [d]$, $x_{\ep\eta}=x_{s(\ep)}$. Let
$y=T^{p_0+\sum_{i\in\eta}p_{i+d}}x$. Then, for any $\ep\subset [d]$,
$T^{\sum_{i\in\ep}p_i}y=T^{p_0+\bf p\cdot(\ep\eta)}x\in
U_{\ep\eta}=U_{s(\ep)}$. So
$$\bigcap\limits_{\ep\subset [d]}T^{-\sum_{i\in\ep}p_i}U_{s(\ep)}\neq\emptyset,$$
and $F$ belongs to $\Ind(U_1,U_2)$.

\medskip
Now assume that $(x_1,x_2)$ is an ${\rm Ind}_{fip}$-pair. That is,
for any neighborhood $U_1\times U_2$ of $(x_1,x_2)$, any $d\in \N$
and any $s\in \{1,2\}^{[d]}$, there are positive integers
$p_1,\ldots,p_d$ such that $\bigcap\limits_{\ep\subset
[d]}T^{-\sum_{i\in\ep}p_i}U_{s(\ep)}\neq\emptyset.$ Let $x\in
\bigcap\limits_{\ep\subset [d]}T^{-\sum_{i\in\ep}p_i}U_{s(\ep)}.$
Then $T^{\sum_{i\in\ep}p_i}x\in U_{s(\ep)}$ for any $\ep\subset
[d].$ This implies that $\{x_1,x_2\}^{[d]}\subset \Q^{[d]}$.
\end{proof}

The following lemma is a useful application of the previous lemma.

\begin{lem}\label{lem-Q}
Let $(X,T)$ be a transitive system, $x_1 \in X$ be a transitive point and
$d\geq 1$ be an integer. Suppose that $(x_2,x_1,\ldots,x_1)\in\Q^{[d]}$
for some $x_2\in X$ and that $\pi_1:A\rightarrow X$ is semi-open, where
$A=\overline{orb((x_1,x_2),T\times T)}$ and $\pi_1$ is the
projection to the first coordinate. Then $\{x_1,x_2\}^{[d]}\subset
\Q^{[d]}$.
\end{lem}
\begin{proof} If $X$ is finite, then the lemma holds. Thus we assume
that $X$ is infinite. We first prove the following claim.

\medskip

\noindent {\bf Claim:} {\em If ${\bf x}=(x_1,{\bf
a}_*)\in\{x_1,x_2\}^{[d]} \cap\Q^{[d]}$, then $(x_2,{\bf
a}_*)\in\Q^{[d]}$.}

\medskip

Let $U_1$ and $U_2$ be neighborhoods of $x_1$ and $x_2$
respectively.  Since $\pi_1$ is semi-open and $X$ is infinite, then
$V_1=int(\pi_1(U_1\times U_2\cap A))\neq\emptyset$ and infinite. Set
$V_2=U_2$.


Let $s=(s(\ep) : \ep\subset[d] )\in\{1,2\}^{[d]}$
such that ${\bf x}=(x_{s(\ep)}:\ep\subset[d])$. From the hypothesis, ${\bf
x}=(x_1,{\bf a}_*)\in\{x_1,x_2\}^{[d]} \cap\Q^{[d]}$, then there exist
positive integers $n_0,n_1,\ldots,n_d$ such that $T^{n_0+{\bf
n}\cdot\ep}x_1\in V_{s(\ep)}$ for each $\ep\subset [d]$, i.e.
$\bigcap_{\ep\subset[d]}T^{-{\bf
n}\cdot\ep}V_{s(\ep)}\neq\emptyset$, where ${\bf
n}=(n_1,\ldots,n_d)$.

Let $W_1=\bigcap_{\ep\subset[d]}T^{-{\bf
n}\cdot\ep}V_{s(\ep)}\subset V_1$  and $W_2=V_2=U_2$. Since $W_1\times
W_2\cap A\neq\emptyset$, there exists a positive integer $M$ such
that $(T\times T)^{M}(x_1,x_2) \in W_1\times W_2$. And, since
$(x_2,x_1,\ldots,x_1)\in\Q^{[d]}$, there exist positive integers
$m_0,m_1,\ldots,m_d$ such that $T^{m_0}x_1\in
T^{-M}W_2=T^{-M}U_2$, and $T^{m_0+{\bf m}\cdot\eta}x_1\in
T^{-M}W_1$ for all $\eta\subset[d]\setminus\{\emptyset\}$, where
${\bf m}=(m_1,\ldots,m_d)$. It follows that
$T^{-m_0-M}U_2\cap\bigcap_{\eta\subset[d]\setminus\{\emptyset\}}T^{-m_0-M-{\bf
m}\cdot\eta}W_1 \neq\emptyset.$ That is,

$$U_2\cap\bigcap_{\eta\subset[d]\setminus\{\emptyset\}}T^{-{\bf m}\cdot\eta}\bigcap_{\ep\subset[d]}
T^{-{\bf n}\cdot\ep}V_{s(\ep)} \neq\emptyset.$$

Thus,
$$U_2\cap\bigcap_{\eta\subset[d]\setminus\{\emptyset\}}T^{-{\bf m}\cdot\eta}\bigcap_{\ep\subset[d]}
T^{-{\bf n}\cdot\ep}U_{s(\ep)} \neq\emptyset.$$ Set $p_i=n_i+m_i$
for $1\leq i\leq d$, and ${\bf p}=(p_1,\ldots,p_d)$. Then we have
$$U_2\cap \bigcap_{\eta\subset[d]\setminus\{\emptyset\}}T^{-{\bf p}\cdot\eta}U_{s(\eta)}\neq\emptyset,$$
and so we get that $(x_2,{\bf a}_*)\in\Q^{[d]}$. The proof of the
claim is completed.

\medskip

For any ${\bf x}\in\{x_1,x_2\}^{[d]}$, let $l({\bf x})$ be the
number of $x_2$'s appearing in ${\bf x}$. We prove this lemma by
induction on $l({\bf x})$.

If $l({\bf x})=0$, then obviously ${\bf x}=x_1^{[d]}\in\Q^{[d]}$.
Suppose the lemma holds when $l({\bf x})\leq k$, i.e.
${\bf x}\in\Q^{[d]}$ if $l({\bf x})\leq k$. Now, for $l({\bf x})=k+1$
without loss of generality we write ${\bf x}=(x_2,{\bf a}_*)$. Since
$l((x_1,{\bf a}_*))=k$, we have by hypothesis that $(x_1,{\bf a}_*)\in\Q^{[d]}$.
Thus from the claim we get that ${\bf x}=(x_2,{\bf a}_*)\in\Q^{[d]}$.
The proof of this lemma is completed.
\end{proof}

The following corollary extends Corollaries 4.2. and 4.3. from \cite{HKM}, and the
comment right after, that were only proved in the
distal case.

\begin{cor}\label{cor-Q}
Let $(X,T)$ be a minimal t.d.s., $x_1,x_2\in X$ and $d\geq 1$ an integer. If
$(x_1,x_2)\in \RP^{[d]}$ and $(x_1,x_2)$ is a $T\times T$-minimal
point, then $\{x_1,x_2\}^{[d+1]}\subset \Q^{[d+1]}$.
\end{cor}

\begin{proof}
Let $A=\overline{orb((x_1,x_2),T\times T)}$. Since $(A,T\times T)$ and
$(X,T)$ are minimal, the projection $\pi_1:A\rightarrow X$ is
semi-open. By Theorem \ref{thm-1},
$(x_2,x_1,\ldots,x_1)\in\Q^{[d+1]}$, so $\{x_1,x_2\}^{[d+1]}\subset
\Q^{[d+1]}$ by Lemma \ref{lem-Q}.
\end{proof}

By the above discussion, we get the following criterion to be a ${\rm
Ind}_{fip}$-pair.

\begin{cor}\label{inde}
Let $(X,T)$ be a minimal system and $(x_1,x_2)\in\RP^{[\infty]}
\setminus\Delta_X=\bigcap_{d\geq1}\RP^{[d]}\setminus\Delta_X$. If
$(x_1,x_2)$ is $T\times T$-minimal or the projection
$$\pi_1:\overline{orb((x_1,x_2),T\times T)}\rightarrow X$$ is
semi-open, then $(x_1,x_2)$ is an ${\rm Ind}_{fip}$-pair.
\end{cor}

\begin{proof}
If $(x_1,x_2)$ is $T\times T$-minimal, by Corollary \ref{cor-Q}, we
have $\{x_1,x_2\}^{[d]}\subset\Q^{[d]}$ for every integer $d\geq1$.
Now, if the projection $\pi_1:\overline{orb((x_1,x_2),T\times T)}\rightarrow X$
is semi-open, then, since $(x_1,x_2)\in\RP^{[\infty]}$, we have
$(x_2,x_1,\ldots,x_1)\in\Q^{[d]}$ for every $d\geq1$. Hence, by Lemma \ref{lem-Q},
we also have $\{x_1,x_2\}^{[d]}\subset\Q^{[d]}$ for
every integer $d\geq1$.

By Lemma \ref{lem}, the proof is completed.
\end{proof}

\subsection{The structure of minimal systems without nontrivial
${\rm Ind}_{fip}$-pairs}

The following is the main result of this section.

\begin{thm}\label{prop11}
Let $(X, T)$ be a minimal system. If $X$ does not contain any
nontrivial ${\rm Ind}_{fip}$-pair, then it is an almost one-to-one
extension of its maximal $\infty$-step nilfactor.
\end{thm}

To prove this theorem we need some preparation. Every extension of
minimal systems can be lifted to an open extension by almost
one-to-one modifications. To be precise, for every extension
$\pi:(X,T) \rightarrow (Y,T)$ between minimal systems there exists a
canonically defined commutative diagram of extensions (called the
{\it shadow diagram})

\[
\begin{CD}
X^* @>{\sigma}>> X\\
@V{\pi^*}VV      @VV{\pi}V\\
Y^* @>{\tau}>> Y
\end{CD}
\]

\medskip

with the following properties:
\begin{enumerate}
  \item $\sigma$ and $\tau$ are almost one-to-one extensions;
  \item $\pi^*$ is an open extension, i.e. for any open set $U\subset X^*$,
  $\pi^*(U)$ is an open set of $Y^*$;
  \item $X^*$ is the unique minimal set in $R_{\pi\tau}=\{(x,y)\in
  X\times Y^*:\pi(x)=\tau(y)\}$ and $\sigma$ and $\pi^*$ are the
  restrictions to $X^*$ of the projections of $X\times Y^*$ onto
  $X$ and $Y^*$ respectively.
\end{enumerate}

We refer to \cite{Au88, G76, V70, V77} for the details of this
construction.

\medskip

In \cite{G76} it was shown that, a metric minimal system $(X,T)$ with
the property that $n$-proximal tuples are dense in $X^n$ for every
$n \ge 2$, is weakly mixing. This was extended by van der Woude
\cite{W} as follows (see also \cite{G05}).

\begin{thm}\label{W}
Let $\pi:(X,T) \rightarrow (Y,T) $ be a factor map between the metric minimal systems
$(X,T)$ and $(Y,T)$. Suppose that $\pi$ is an open proximal extension,
then $\pi$ is a weakly mixing extension, i.e. $(R_\pi, T\times T)$ is
transitive.
\end{thm}

Now we give the proof of Theorem \ref{prop11}.

\medskip

\noindent {\bf Proof of Theorem \ref{prop11}.}\quad
Let $(X, T)$ be a minimal system without ${\rm Ind}_{fip}$-pairs.
We denote by $(Y, T)$ the quotient system of $(X,T)$ determined by the
equivalence relation $\RP^{[\infty]}(X,T)$ and let $\pi:(X,T)\rightarrow (Y,T)$
be the canonical projection map.

We first prove that $\pi$ is a proximal extension. Remark that if
$(x,y)\in R_{\pi}=\RP^{[\infty]}$ is a $T\times T$ minimal point,
according to Corollary \ref{inde}, we have $(x,y)$ is an ${\rm
Ind}_{fip}$-pair and thus we must have $x=y$. Now consider any
$(x,y)\in R_{\pi}$ and $u\in E(X,T)$ a minimal idempotent. Since
$(ux,uy)$ is a $T\times T$ minimal point, we have from previous
observation that $ux=uy$, which implies $(x,y)$ is a proximal pair.

As the shadow diagram shows, there exists a canonically defined
commutative diagram of extensions

\[
\begin{CD}
X^* @>{\sigma}>> X\\
@V{\pi^*}VV      @VV{\pi}V\\
Y^* @>{\tau}>> Y
\end{CD}
\]

\medskip

\noindent verifying properties (1)-(3) above.

Since $\pi$ is proximal and $\sigma$ is almost one-to-one, we have
$\pi\circ\sigma$ is proximal. For any $(x,x')\in R_{\pi^*}$,
$\pi\circ\sigma(x)=\pi\circ\sigma(x')$, so $(x,x')$ is a proximal
pair, which implies $\pi^*$ is proximal too. By Theorem \ref{W},
$\pi^*$ is a weakly mixing extension, and hence there exists
$(x_1,x_2)\in R_{\pi^*}$ such that
$R_{\pi^*}=\overline{orb((x_1,x_2),T\times T)}$. Let $\pi_1$ be the
projection of $R_{\pi^*}$ to the first coordinate. It is easy to get
that for any open sets $U, V\subset X^*$, $\pi^*(U\times
V)=\pi^*(U)\cap\pi^*(V)$ is an open set. So we get that $\pi_1$ is
an open map too. Since $(x_1,x_2)$ is a proximal pair, we have
$(x_1,x_2)$ is an ${\rm Ind}_{fip}$-pair by Corollary \ref{inde}.
Hence $(\sigma(x_1),\sigma(x_2))$ is an ${\rm Ind}_{fip}$-pair
too. Then we must have $\sigma(x_1)=\sigma(x_2)$, and thus
$R_{\pi^*}\subset R_\sigma$.

Since $\tau$ is almost one-to-one, we can choose $y\in Y$ such
that $\tau^{-1}(y)$ contains only one point. Suppose $x_1,x_2\in\pi^{-1}(y)$,
then there exist $x_1^*,x_2^*\in X^*$ such that $\sigma(x_1^*)=x_1$
and $\sigma(x_2^*)=x_2$. As $\tau\circ\pi^*(x_1^*)=\pi\circ\sigma(x_1^*)=y=
\pi\circ\sigma(x_2^*)=\tau\circ\pi^*(x_2^*)$, we have
$\pi^*(x_1^*),\pi^*(x_2^*)\in\tau^{-1}(y)$ and so $\pi^*(x_1^*)=\pi^*(x_2^*)$,
i.e. $(x_1^*,x_2^*)\in R_{\pi^*}\subset R_\sigma$. Hence $x_1=x_2$,
which implies that $\pi^{-1}(y)$ contains only one point too.
We conclude $\pi$ is almost one-to-one.
\hfill $\square$

\begin{prop}\label{distal} Let $(X,T)$ be a minimal distal system.
Then there are no nontrivial ${\rm Ind}_{fip}$-pairs if and only if
$(X,T)$ is an $\infty$-step nilsystem.
\end{prop}

\begin{proof}
It is a direct consequence of Theorem \ref{prop11}.
\end{proof}

\subsection{The assumption of semi-openness in Lemma \ref{lem-Q} cannot be removed}

In this subsection we give an example to show that the condition
of semi-openness in Lemma \ref{lem-Q} cannot be removed. First we recall some notions.

Let $(X,T)$ be t.d.s. For any open cover $\U$, let $N(\U)$ denote the
smallest possible cardinality among finite subcovers of $\U$.
Given an increasing sequence
$\A=\{t_1,t_2,\ldots\}$ in $\Z_+$, the sequence entropy of $(X,T)$ or just $T$
with respect to $\A$ and the cover $\U$ is
$$h_{\A}(T,\U)=\limsup\limits_{n\rightarrow\infty}\frac{1}{n}\log
N\left(\bigvee\limits_{i=1}^n T^{-t_i}\U\right)$$ and the sequence
entropy of $T$ with respect to $\A$ is
$h_{\A}(T)=\sup_{\U}h_{\A}(T,\U)$, where the supremum is taken over
all finite open covers $\U$ of $X$. The system $(X,T)$ is a {\em
null system} if for any sequence $\A\subset\Z_+$, $h_{\A}(T)=0$.

Similar to ${\rm Ind}_{fip}$-pair we can define IN-pairs. Let
$(X,T)$ be t.d.s. and $(x_1,x_2)\in X\times X$. Then, $(x_1,x_2)$ is
a {\em IN-pair} if for any neighborhoods $U_1$, $U_2$ of $x_1$,
$x_2$ respectively, $Ind(U_1,U_2)$ contains arbitrary long finite
independence sets. In \cite{KL} is shown that $(X,T)$ is a null
system if and only if it contains no nontrivial IN-pairs. It is
obvious that a null system contains no nontrivial
${\rm Ind}_{fip}$-pairs as ${\rm Ind}_{fip}$-pairs are IN-pairs.

\medskip

The following example is classical.

\begin{ex} Sturmian system.
\end{ex}
Let $\a$ be an irrational number in the interval $(0,1)$ and
$R_\a$ be the irrational rotation on the (complex) unit circle $\T$ generated
by $e^{2\pi i\a}$. Set
$$A_0=\left\{e^{2\pi i\theta}:0\leq \theta < (1-\a) \right\} \text{ and }
A_1=\left\{e^{2\pi i\theta}: (1-\a) \leq \theta < 1\right\}$$

Consider $z \in \T$ and define $x \in \{0,1\}^{\Z}$ by: for all
$n\in \Z$, $x_{n}=i$ if and only if $R_{\a}^n(z) \in A_{i}$. Let  $X
\subset \{0,1\}^{\Z}$ be the orbit closure of $x$ under the shift
map $\sigma$ on $\{0,1\}^{\Z}$, i.e. for any $y \in \{0,1\}^{\Z}$,
$(\sigma(y))_n=y_{n+1}$. This system is called Sturmian system. It
is well known that $(X,\sigma)$ is a minimal almost one-to-one
extension of $(\T,R_\a)$. Moreover, it is an asymptotic extension.
Also, it is not hard to prove that it is a null system.

Let $\pi: X \to \T$ be the former extension and consider
$(x_1,x_2)\in R_{\pi} \setminus \Delta_X$. Then $(x_1,x_2)$ is an
asymptotic pair and thus $(x_1,x_2)\in \RP^{[d]}$ for any integer
$d\geq 1$. In particular, $(x_{2},x_{1},\ldots,x_{1}) \in
\Q^{[d+1]}$ for any integer $d\geq 1$. Now, if
$\{x_1,x_2\}^{[d+1]}\subset \Q^{[d+1]}$ for any integer $d\geq
1$, then, by Lemma \ref{lem}, $(x_1,x_2)$ is a nontrivial ${\rm
Ind}_{fip}$-pair, and so an IN-pair, which contradicts the fact that
$(X,T)$ is a null system. Therefore, there exists an integer $d\geq
1$ such that $\{x_1,x_2\}^{[d]}\nsubseteq\Q^{[d]}$, which implies
the condition $\pi_1:\overline{orb((x,y),T\times T)}\rightarrow X$
is semi-open in Lemma \ref{lem-Q} cannot be removed.

\section{Minimal distal systems which are not $\infty$-step nilsystems}\label{section-example}

In the previous section we showed that a minimal distal system is an
$\infty$-step nilsystem if and only if it has no nontrivial ${\rm
Ind}_{fip}$-pairs. In this section we will give examples of minimal
distal systems which are not $\infty$-step nilsystems. We remark
that if $(X,T)$ is minimal distal and $\pi:(X,T) \lra (X_{eq},T)$ is the
maximal equicontinuous factor of $(X,T)$, then each
pair in $R_{\pi}\setminus \Delta_{X}$ is an untame pair (see \cite{HLSY,
H1, MS}). In fact, the result in previous section tells that if
$\pi_{\infty}:(X,T) \lra (Z_\infty,T)$ is the factor map from $X$ to its
maximal $\infty$-step nilfactor, then each pair in
$R_{\pi_\infty}\setminus \Delta_{X}$ is  an ${\rm Ind}_{fip}$-pair. We
do not know how to glue both results together.

\subsection{The existence}

To show the existence of minimal distal systems which are not
$\infty$-step nilsystems we use some abstract results from \cite{L}.
We use freely notations therein.

\begin{prop}\cite[Theorem 4.4]{L}\label{L1}
Every ergodic measurable distal system \hfill\break $(X,\mathcal{B},\mu,T)$
can be represented as a minimal topologically distal system equipped
with a Borel measure of full support.
\end{prop}

\begin{lem}\cite[Claim 5.5]{L}\label{L2}
Suppose $(Y, \mathcal{D}, \nu, T)$ is an isometric extension of the ergodic rotation
$(\T, \mathcal{B}, \lambda, R_\alpha)$ by a finite group
that is not a Kronecker system. Then $(Y, \mathcal{D},\nu, T)$
does not have a uniquely ergodic distal model.
\end{lem}

The following result produces as much examples as we want.

\begin{thm}
Suppose $(Y, \mathcal{D},\nu, T)$ is an isometric extension of
the ergodic rotation $(\T, \mathcal{B}, \lambda, R_\alpha)$ by a finite group,
that is not a Kronecker system. Then any minimal distal topological
model of $(Y, \mathcal{D},\nu, T)$ is not an $\infty$-step
nilsystem.
\end{thm}
\begin{proof} By Proposition \ref{L1}, $(Y, \mathcal{D},\nu, T)$ has
a minimal topologically distal system $(X,T)$ model equipped with a Borel
measure of full support. By Lemma \ref{L2}, $(X,T)$ cannot be
uniquely ergodic. It is clear that a minimal $\infty$-step
nilsystem is uniquely ergodic, so $(X,T)$ is not an $\infty$-step
nilsystem.
\end{proof}

\subsection{An explicit example}
A way to produce an explicit example is to use the following
Furstenberg result. It appeared first in \cite{F61}. We recall that
a topological dynamical system is \emph{strictly ergodic} if it is
minimal and uniquely ergodic.

\begin{thm}\cite[Theorem 3.1]{F61}\label{se}
Let $(\W_0, T_0)$ be a strictly ergodic system and $\mu_0$ its
unique ergodic measure. Let $\W=\W_0\times \T$ and let
$T:\W\rightarrow \W$ be defined by $T(\w_0,s)=(T_0(\w_0),g(\w_0)s)$,
where $g: \W_0 \rightarrow \T$ is a continuous function. Then, if
the equation
\begin{equation}
g^k(\w_0)=R(T_0(\w_0))/R(\w_0)
\end{equation}
has a solution $R: \W_0 \rightarrow \T$ which is measurable but not
equal almost everywhere to a continuous function, then
$\displaystyle
\lim_{N\rightarrow\infty}\frac{1}{N}\sum_{n=0}^{N-1}f\circ
T^{n}(\w)$ cannot exist for all continuous functions $f$ and all
$\w\in\W$.
\end{thm}

Now we recall some elements of the example from \cite{F61}
satisfying the criterion of the previous theorem. The first step
(that we omit here) is the construction of a sequence of integers
$(n_{k})_{k\in \Z}$ and an irrational number $\a$ such that
$$h(\theta)=\sum_{k\neq 0}\frac{1}{|k|}(e^{2\pi i n_k\a}-1)e^{2\pi i n_k\theta}$$
and $g(e^{2\pi i\theta})=e^{2\pi i\lambda h(\theta)}$, where
$\lambda$ is as yet undetermined, are $C^{\infty}$ functions of
$[0,1)$ and $\T$ respectively. Clearly,
$h(\theta)=H(\theta+\a)-H(\theta)$, where
$$H(\theta)=\sum_{k\neq 0}\frac{1}{|k|}e^{2\pi i n_k\theta} \ .$$
Thus $H(\cdot)$ is in $L^2(0,1)$ and in particular defines a
measurable function. However, $H(\cdot)$ cannot correspond to a
continuous function since $\sum_{k\neq 0}\frac{1}{|k|}=\infty$ and
hence the series is not Ces\`{a}ro summable at $\theta=0$
(\cite{An}). Therefore, for some $\lambda$, $e^{2\pi i\lambda H(\theta)}$ cannot be a
continuous function either. Considering $R(e^{2\pi i\theta})=
e^{2\pi i\lambda H(\theta)}$ we get
$$R(e^{2\pi i\a}s)/R(s)=g(s)$$ with $R: \T \rightarrow \T$ measurable but not continuous.

By Theorem \ref{se}, the transformation
$T$ of $\T \times \T$ given by
$$T(s_1,s_2)=(e^{2\pi i\a}s_1,g(s_1)s_2)$$
will not possess all its ergodic averages, i.e. there are
a continuous function $f$ and $\w \in \T \times \T$ such that
$\displaystyle \lim_{N\rightarrow\infty}\frac{1}{N}\sum_{n=0}^{N-1} f\circ T^n(\w)$ does not
exist.

Let $\W=\ov{orb(\w,T)}\subset \T \times \T$. It is easy to get
that $(\T \times \T,T)$ is distal, and so $(\W,T)$ is minimal
distal. If $(\W,T)$ is an $\infty$-step nilsystem, then it is an
inverse limit of minimal nilsystems, and of course $(\W,T)$ is
strictly ergodic. Let $h=f|_\W$ and $h$ apparently continuous on
$\W$, so $\displaystyle \lim_{N\rightarrow\infty}\frac{1}{N}\sum_{n=0}^{N-1}
g\circ T^n (\w)$ exists, contradicting the choice of $f$ and $\w$. Therefore
$(\W,T)$ is minimal distal but not an $\infty$-step nilsystem.

\section{Discussion about the unique ergodicity}\label{uniqueergo}

In this section we aim to investigate the question whether a minimal
system without ${\rm Ind}_{fip}$-pairs is uniquely ergodic. First we
observe that when $(X,T)$ is minimal, then
$(X_\infty=X/\RP^{[\infty]},T)$ is uniquely ergodic since it is an
inverse limit of uniquely ergodic systems.

\medskip

Assume that $(X,T)$ is a minimal system and let $\mu$ be an ergodic
measure for $(X,T)$. Then $(X, \mathcal{B},\mu, T)$ is an ergodic
measure preserving system, where $\mathcal{B}$ is the
$\sigma$-algebra of Borel sets of $X$ (we omit $\mathcal{B}$ in the
sequel). In \cite{HK05}, to prove the convergence in
$L^2(X,\mathcal{B},\mu)$ of some non-conventional ergodic averages,
the authors introduced measures
$\mu^{[d]}\in\mathcal{M}(X^{[d]},T^{[d]})$ for any integer $d\geq 1$
and  used them to produce the maximal measure theoretical factor of
order $d$ of $(X,\mu,T)$ (in the measurable context this means that
the system is an inverse limit, with measurable factor maps, of
$d$-step nilsystems), denoted by $(Z_d, \cZ_d,\mu_d,T)$.

\medskip

In the topological setting, $(X_d= X/\RP^{[d]},T)$ is the maximal
factor of order $d$ of $(X,T)$ and is uniquely ergodic. In
\cite{HKM} it was observed that $(X_d,T)$ is also a system of order
$d$ in the measurable sense for its unique invariant measure. This
implies that $X_d$ is a factor of $Z_d$ in the measurable sense.

\medskip

Let $\mu=\int_{Z_{d}} \mu_z d\mu_d(z)$ be the disintegration of
$\mu$ over $\mu_d$. Pairs in the support of the measure
$$\lambda_d=\int_{{Z}_{d}} \mu_z\times \mu_z d\mu_d(z)$$
are called $\F_d^\mu$-pairs, where $\F_d^\mu=supp(\lambda_d)$. To
study $\F_d^\mu$ we will need the following lemma from \cite{HK05}.

\begin{lem}\cite[Propostion 4.7., Theorem 13.1.]{HK05}\label{HK}
Let $d\geq1$ be an integer and $V_d=\{0,1\}^{[d]}$.
\begin{enumerate}
\item For $f_\ep$, $\ep\in V_d$, bounded measurable
functions on $X$,
$$\int_{X^{[d]}}\bigotimes\limits_{\ep\in V_d}f_\ep d\mu^{[d]}=
\int_{{Z}_{d-1}^{[d]}}\bigotimes\limits_{\ep\in
V_d}\E(f_\ep|\mathcal{Z}_{d-1}) d\mu_{d-1}^{[d]},$$ where
$({Z}_{d-1}^{[d]},\mu_{d-1}^{[d]})$ is the joining of $2^d$ copies
of $({Z}_{d-1},\mu_{d-1})$.

\item For $f_\ep$, $\ep\in V_d$, bounded measurable
functions on $X$, the average
$$\prod\limits_{i=1}^d\frac{1}{N_i-M_i}\sum\limits_{{\bf n}\in
[M_1,N_1)\times\cdots\times[M_d,N_d)}\int_{X} \prod\limits_{\ep\in V_d}
f_\ep\circ T^{{\bf n}\cdot\ep} d\mu$$
converges to
$$\int_{X^{[d]}}\prod\limits_{\ep\in V_d}f_\ep(x_\ep) d\mu^{[d]}({\bf x})$$
as $N_i-M_i\rightarrow \infty$ for all $1\leq i\leq d$.
\end{enumerate}
\end{lem}
\medskip

We get the following theorem about $\F_d^\mu$.

\begin{thm}\label{weakform}
Let $(X,T)$ be a minimal system and $\mu$ an ergodic measure on $X$.

\begin{enumerate}\item Let $d\geq 1$ be an integer, then $\F_d^\mu\subset \RP^{[d]}$.

\item $\bigcap_{d\in \N} \F_d^\mu\subset {\rm Ind}_{fip}(X,T)$.

\end{enumerate}
\end{thm}

\begin{proof}
Let $(x_0,x_1)\in \F_d^\mu$. Then for any neighborhood $U_0\times U_1$
of $(x_0,x_1)$
$$\lambda_d(U_0\times U_1)=\int_{Z_d} \E(1_{U_0}|\mathcal{Z}_d)\E(1_{U_1}|\mathcal{Z}_d)d \mu_d>0.$$

Let $\{U_\ep:\ep\in V_{d+1}\}\subset\{U_0,U_1\}$ be open sets.
We claim that
$$\int_{X^{[d+1]}}\prod_{\ep\in
V_{d+1}}1_{U_\ep}(x_\ep)d\mu^{[d+1]}({\bf x})>0.$$

\noindent \textit{Proof of the claim}: Since $\int_{Z_d}
\E(1_{U_0}|\mathcal{Z}_d)\E(1_{U_1}|\mathcal{Z}_d)d \mu_d>0$, there
exists $B\in\mathcal{Z}_d$ with $\mu_d(B)>0$ such that for any $z\in
B$, $\E(1_{U_0}|\mathcal{Z}_d)(z)\E(1_{U_1}|\mathcal{Z}_d)(z) > 0$.
Also, since $U_\ep=U_0$ or $U_\ep=U_1$ for $\ep\in V_{d+1}$, we have
for any ${\bf z}\in B^{[d+1]}$ that $\prod_{\ep\in
V_{d+1}}\E(1_{U_\ep}|\mathcal{Z}_d)(z_\ep) >0$. Now, by the previous
lemma, we obtain
\begin{equation*}
\begin{aligned}
\int_{X^{[d+1]}}\prod_{\ep\in
V_{d+1}}1_{U_\ep}(x_\ep)d\mu^{[d+1]}({\bf x})&=
\int_{Z_d^{[d+1]}}\prod_{\ep\in
V_{d+1}}\E(1_{U_\ep}|\mathcal{Z}_d)(z_\ep)d\mu_d^{[d+1]}({\bf z})
\\&\geq \int_{B^{[d+1]}}\prod_{\ep\in V_{d+1}}\E(1_{U_\ep}|\mathcal{Z}_d)(z_\ep)d\mu_d^{[d+1]}({\bf z})
>0.
\end{aligned}
\end{equation*}
This completes the proof of the claim since it was shown in
\cite{HK05} that $\mu_d^{[d+1]}(B^{[d+1]})\ge \mu_d(B)^{2^{d+1}}>0$.

Again, by Lemma \ref{HK},
$$\frac{1}{N^{d+1}}\sum_{0\le n_1,\ldots,n_{d+1}< N} \int_{X}
\prod_{\ep\in V_{d+1}}1_{U_\ep}\circ
T^{n_1\ep_1+\ldots+n_{d+1}\ep_{d+1}}d\mu$$ converges to
$$\int_{X^{[d+1]}}\prod_{\ep\in
V_{d+1}}1_{U_\ep}(x_\ep)d\mu^{[d+1]}({\bf x})>0$$
as $N\rightarrow\infty$. Hence, there exists a Borel set
$B$ with $\mu(B)>0$ and $n_1,\ldots, n_{d+1}\in \Z_{+}$
such that for each $x\in B$
$$\prod_{\ep\in V_{d+1}}1_{U_\ep}\circ
T^{n_1\ep_1+\ldots+n_{d+1}\ep_{d+1}}(x)>0.$$ That is, for each $x\in
B$, $T^{n_1\ep_1+\ldots+n_{d+1}\ep_{d+1}}(x)\in U_\ep$ or
$$\bigcap_{\ep\in
V_{d+1}}T^{-n_1\ep_1-\ldots-n_{d+1}\ep_{d+1}}U_\ep\not=\emptyset.$$

First we prove statement (1). Let $(x_0,x_1)\in\F_d^\mu$. For $\delta>0$
let $U_0\times U_1$ be a neighborhood of $(x_0,x_1)$ where the diameters of
$U_0$ and $U_1$ are less than $\delta$. Set $U_{(0,\ldots,0)}=U_0$,
$U_{(0,\ldots,0,1)}=U_1$ and $U_\ep=U_0$ for any other
$\ep\in V_{d+1}$. By previous discussion, there exist $x\in X$
and $n_1,\ldots, n_{d+1}\in \Z_+$ such that
$T^{n_1\ep_1+\ldots+n_{d+1}\ep_{d+1}}(x)\in U_\ep$
for any $\ep\in V_{d+1}$. Let $y_0=x=T^0 x\in U_{(0,\ldots,0)}=U_0$,
$y_1=T^{n_{d+1}}x\in U_{(0,\ldots,0,1)}=U_1$. Then, for any
$\ep\in V_d\setminus\{(0,\ldots,0)\}$, we have
$T^{n_1\ep_1+\ldots+n_{d}\ep_{d}}(y_0)\in U_{\ep0}=U_0$ and
$T^{n_1\ep_1+\ldots+n_{d}\ep_{d}}(y_1)\in U_{\ep1}=U_0$.
Since the diameters of $U_0$ and $U_1$ are less than $\delta$, we get that
$d(x_0,y_0)<\delta$, $d(x_1,y_1)<\delta$ and for any
$\ep\in V_d\setminus\{(0,\ldots,0)\}$, $d(T^{n_1\ep_1+\ldots+n_{d}\ep_{d}}(y_0),
T^{n_1\ep_1+\ldots+n_{d}\ep_{d}}(y_1))<\delta$.
Therefore $\F_d^\mu\subset \RP^{[d]}$.

\medskip

To show (2), by Lemma \ref{lem}, it remains to prove that if
$(x_0,x_1)\in\bigcap_{d\in \N} \F_d^\mu$ then
$\{x_0,x_1\}^{[d]}\subset\Q^{[d]}$ for any integer $d\geq 1$.
Let $d\geq 1$ be an integer and take ${\bf x}=(x_\ep)_{\ep\in
V_d}\in\{x_0,x_1\}^{[d]}$. Given a neighborhood ${\bf V}$ of ${\bf
x}$, there exists a neighborhood $U_0 \times U_1$ of $(x_0,x_1)$
such that if we set $U_\ep=U_i$ depending on
$x_\ep=x_0$ or $x_\ep=x_1$, then $\bigotimes_{\ep\in V_d}U_\ep\subset
{\bf V}$. From the conclusion of part (1), there exist $x\in X$ and
$n_1,\ldots, n_d\in \Z_+$ such that
$T^{n_1\ep_1+\ldots+n_d\ep_d}(x)\in U_\ep$ for any $\ep\in V_d$. Let
${\bf n}=(n_1,\ldots,n_d)$. We have $(T^{{\bf n}\cdot\ep}x)_{\ep\in
V_d}\in{\bf V}$, and thus ${\bf x}\in\Q^{[d]}$. The proof is
completed.
\end{proof}

\begin{rem}
We have shown that $\F_d^\mu\subset \RP^{[d]}$. However, the
converse is not true in general. For example, let $(Z,S)$ be a
non-trivial $d$-step nilsystem  and $\nu$ an ergodic measure on $Z$.
In \cite{We} it was shown that there exists a weakly mixing minimal
uniquely ergodic system $(X,T)$ with the uniquely ergodic measure
$\mu$ satisfying that for all $x, y \in X$ with $y\not\in\{T^n
x\}_{n\in \Z}$, the orbit $\{(T^n x,T^n y)\}_{n\in \Z}$ is dense in
$X\times X$; and  $(Z,\nu,S)$ and $(X,\mu,T)$ are isomorphic. Then
$(X, \mu)$ coincides with $(Z_d(X),\mu_d)$, and so
$\F_\mu^d=\Delta_X$. Since $(X,T)$ is weakly mixing, we get that
$\RP^{[d]}=X\times X$, which implies that
$\F_\mu^d\subsetneqq\RP^{[d]}$.
\end{rem}

\medskip

A direct application of the above theorem is the following.

\begin{thm} Let $(X,T)$ be a minimal system with
${\rm Ind}_{fip}(X,T)=\Delta_X$, then for each
ergodic measure $\mu$, $(X, \mu,T)$ is measure theoretical isomorphic
to an $\infty$-step nilsystem.
\end{thm}
\begin{proof} Applying Theorem \ref{weakform} we get that $\bigcap_{d\in \N} \F_d^\mu\subset {\rm Ind}_{fip}(X,T)=\Delta_X$.
The result follows, since it is easy to check that $\bigcap_{d\in
\N} \F_d^\mu=\F_{\infty}^\mu$, where $\F_{\infty}^\mu$ is the
support of $\lambda_\infty=\int_{\mathcal{Z}_{\infty}} \mu_z\times
\mu_z d\mu_\infty(z)$ with $({Z}_{\infty},\mu_\infty)$ the inverse
limit of $({Z}_{d},\mu_d).$
\end{proof}

To show the unique ergodicity of an $\infty$-step nilsystem the following question is crucial.

\begin{ques}\label{conj2}
Let $(X,T)$ be an $E$-system (i.e. is transitive and admits an
invariant measure with full support), let $x$ be a transitive point
and $p$ be a fixed point of $(X,T)$. Is it true that
$(p,x,\ldots,x)\in \Q^{[d]}$ for any integer $d\geq 1$ ?
\end{ques}

If this question has a positive answer, then by Lemma \ref{lem-Q},
we have $\{x,p\}^{[d]}\subset\Q^{[d]}$ for any integer $d\geq 1$
(since $\overline{orb((x,p),T\times T)}=X\times \{p\}$), and thus
$(x,p)$ is an ${\rm Ind}_{fip}$-pair by Lemma \ref{lem}.

\begin{conj} Let $(X,T)$ be a minimal system with
${\rm Ind}_{fip}(X,T)=\Delta_X$. Then $(X,T)$ is uniquely ergodic.
\end{conj}

If Question \ref{conj2} has a positive answer, then using the proof
of \cite[Theorem 4.4]{HLSY} and the lifting property of ${\rm
Ind}_{fip}$-pairs \cite{HLY1}, we may conclude that the conjecture
holds.

\section{Topological complexity of $\infty$-step nilsystems}\label{section-complexity}

The big development in the study of non-conventional ergodic
averages during the last decade has put in evidence, among other
facts, the crucial role of nilsystems when studying ``polynomial''
phenomena in dynamical systems theory. The objective of this section
is to prove that $\infty$-step nilsystems have polynomial topological complexity.
It is well known that bounded complexity characterize minimal
rotations on compact Abelian groups (see \cite{BHM}). Some basic symbolic examples
(substitutions systems for example) show that polynomial complexity cannot
characterize $\infty$-step nilsystems, so  the characterizion of polynomial complexity seems to be a deep problem far to be solved.

In order to study the quantitative distribution of polynomial orbits
in nilmanifolds, Green and Tao introduced in \cite{GT} a metric
induced by the Mal'cev basis on a nilmanifold and they studied its
behaviour under left multiplication. We obtain as an application a
polynomial bound for the topological complexity of nilsystems and
consequently of $\infty$-step nilsystems. We use freely the
notations and results from \cite{GT} and \cite{Au}.

\subsection{Polynomial behaviour of orbits}
In this subsection we assume $G$ is a connected and simply connected Lie
Group and $\Gamma \subset G$ is a co-compact subgroup.
We will denote by $G_i$ the $i$-th subgroup of the
associated lower central series. Under these assumptions Mal'cev
proved the following theorem.

\begin{thm}[Mal'cev basis]
Let $G$ be an $m$-dimensional nilpotent Lie group and $\Gamma \subset
G$ a co-compact subgroup. There exists a basis $\mathcal{X}=\{
X_1,\ldots ,X_m\}$ of the associated Lie algebra $\mathfrak{g}$ such
that:

\begin{enumerate}
\item
For each $j\in \{0,\ldots,m-1\}$ the subspace
$\mathfrak{h}_j=\text{Span}(X_{j+1},\ldots,X_{m})$ is an ideal in
$\mathfrak{g}$ and $exp(\mathfrak{h}_j)$ is a normal subgroup in
$G$.

\item
$G_i=exp(\mathfrak{h}_{m-m_i})$, where $m_i=dim(G_i)$.

\item
Each $g\in G$ can be written uniquely as $exp(t_1X_1)\cdots
exp(t_mX_m)$ where  $t=(t_1,\ldots,t_m) \in \R^m$.

\item  $\Gamma=\{\exp(n_1X_1)\cdots \exp(n_m X_m) : n=(n_1,\ldots,n_m)\in \Z^m\}$.
\end{enumerate}

We say that $\mathcal{X}=\{ X_1,\ldots ,X_m\}$ is a Mal'cev basis
for $G/\Gamma$ adapted to the lower central series $(G_i)_{i\geq
0}$.
\end{thm}

Let $g=\exp(t_1X_1) \cdots \exp(t_mX_m) \in G$ and denote
$\psi(g)=(t_1,\ldots,t_m) \in \R^{m}$. Let
$|\psi(g)|=\|\psi(g)\|_{\infty}$. Fix a Mal'cev basis $\mathcal{X}$.
In \cite{GT}, the authors introduced the following metric on $G$ and $G/\Gamma$.

\begin{defn}
Let $x,y \in G$ and define
$$
d(x,y)=\inf \left \{ \sum_{i=0}^{n-1} min(|\psi(x_{i-1}x_i^{-1})|, |\psi(x_{i}x_{i-1}^{-1})|) \ :
 \  n\in \N, \ x_0,\ldots ,x_n \in G, \ x_0=x, x_n=y \right \}
$$
which is right-invariant, i.e. $d(x,y)=d(xg,yg) \text{ for all } g \in G$ and
$d(x,y)\leq |\psi(xy^{-1})|$.

This metric induces a metric on $G / \Gamma$ that we also call $d(\cdot,\cdot)$ by setting:
\begin{align*}
d(x\Gamma,y\Gamma)&=\inf \{ d(x',y'):  x'\in x\Gamma ,y' \in y\Gamma \} \\
&= \inf\{d(x\gamma_1,y\gamma_2): \gamma_1,\gamma_2 \in \Gamma\}=\inf\{d(x,y\gamma): \gamma \in \Gamma\}
\end{align*}
where the last equality follows from the right-invariance of the metric.
\end{defn}

In the sequel we use some results obtained in \cite{GT} and we rephrase some others
in a convenient way. The Mal'cev basis $\mathcal{X}$ is fixed.

\begin{lem}[Multiplication and inversion] \label{coorpol}
Let $x,y\in G$, $t=\psi(x)$, $u=\psi(y)$. Then,

\begin{enumerate}

 \item

$\psi(xy)=(t_1+u_1,t_2+u_2+P_1(t_1,u_1),\ldots,t_m+u_m+P_{m-1}(t_1,\ldots,t_{m-1},u_1,\ldots,u_{m-1}))$
where for each $i\in\{1,\ldots,m-1\}$, $P_i$ is a real polynomial.

\item $\psi(x^{-1})=(-t_1,-t_2+Q_1(t_1),\ldots,-t_m+Q_{m-1}(t_1,\ldots,t_{m-1}))$
where for each $i\in\{1,\ldots,m\}$, $Q_i$ is a real polynomial.

\end{enumerate}

\end{lem}

We get easily that  $\psi(xy^{-1}) = (R_1(t,u),R_2(t,u),\dots,R_m(t,u))$
where $R_i$ is a real polynomial for each $i\in\{1,\ldots,m\}$. In
order to simplify notations in what follows we will write $P$ to refer to any
generic real polynomial with positive coefficients, not
necessarily the same. This will be clear from the context.

\begin{lem}[Coordinates polynomial bound] \label{coordinates_bound}
Let $x,y \in G$. Then,
  $$d(x,y)\leq  P(\max(|\psi(x)|,|\psi(y)|)) \ |\psi(x)-\psi(y)| $$

\end{lem}

\begin{proof}

Let $\psi(x)=t$ and $\psi(y)=u$. From the definition of the distance and
Lemma \ref{coorpol} we have,
$$d(x,y)\leq |\psi(xy^{-1})|=|(R_1(t,u),R_2(t,u),\ldots,R_m(t,u))|.$$
Write $R_i(t,u)=\sum\limits_{\vec{\alpha},\vec{\beta}}
C_{\vec{\alpha},\vec{\beta}}^{(i)}t^{\vec{\alpha}}u^{\vec{\beta}}$.
Since $R_i(t,t)=0$, one deduces
$$R_i(t,u)=R_i(t,u)-R_i(t,t) =\sum\limits_{\vec{\alpha},\vec{\beta}} C_{\vec{\alpha},
\vec{\beta}}^{(i)} t^{\vec{\alpha}}(u^{\vec{\beta}}-t^{\vec{\beta}})$$
Expanding  $(u^{\vec{\beta}}-t^{\vec{\beta}})=\sum_{i=1}^m
(u_i-t_i)W_{\vec{\beta},i}(t,u)$, where $W_{\vec{\beta},i}$ are
polynomials, we get,
$$d(x,y)\leq |\psi(x)-\psi(y)|\sum\limits_{\vec{\alpha},\vec{\beta}}
\sum\limits_{i=1}^m |C_{\vec{\alpha},\vec{\beta}}^{(i)}||t^{\vec{\alpha}}
||W_{\vec{\beta},i}(t,u)|\leq |\psi(x)-\psi(y)| \ P(max(|\psi(x)|,|\psi(y)|)).$$
\end{proof}

\begin{lem} \label{polybound left}
Let $g,x,y\in G$, then
$$d(gx,gy)\leq P(|\psi(g)|,|\psi(x)|,|\psi(y)|) \ d(x,y)$$
\end{lem}
\begin{proof}
Let $g,z \in G$. From Lemma \ref{coorpol} we see that
$\psi(gzg^{-1})$ is a polynomial function of $\psi(z)$ and $\psi(g)$
that vanishes when $\psi(z)=0$. This this polynomial function can be written as
$\psi(z)P(\psi(g),\psi(z))$. Then,
\begin{equation} \label{ec1}
|\psi(gzg^{-1})|\leq |\psi(z)| \ P(|\psi(z)|,|\psi(g)|)
\end{equation}

By Lemma \ref{coordinates_bound}, $d(x,y)\leq
P(|\psi(x)|,|\psi(y)|)$, and then for computing the distance between
$x$ and $y$ we can restrict to paths $x=x_0, \ldots , x_n=y$ satisfying
$k(x_i,x_{i+1})\leq P(|\psi(x)|,|\psi(y)|)$ where
$k(x_i,x_{i+1})=\min(|\psi(x_{i}x_{i+1}^{-1})|,|\psi(x_{i+1}x_{i}^{-1})|)$.
Let us observe (using Lemma \ref{coorpol}) that this property
implies that
$\max(|\psi(x_{i}x_{i+1}^{-1})|,|\psi(x_{i+1}x_{i}^{-1})|)\leq
P(|\psi(x)|,|\psi(y)|).$

Consider $x=x_0,\ldots,x_n=y$ such a path and the path $gx=gx_0,
gx_1,\ldots,gx_n=gy$. From \eqref{ec1},
\begin{align*}
k(gx_i,gx_{i+1})&\leq |\psi(gx_ix_{i+1}^{-1}g^{-1})|\\
&\leq|\psi(x_ix_{i+1}^{-1})|P(|\psi(x_ix_{i+1}^{-1})|,|\psi(g)|) \\
&\leq|\psi(x_ix_{i+1}^{-1})|P(|\psi(x)|,|\psi(y)|,|\psi(g)|)
\end{align*}

In the same way, $$|\psi(gx_{i+1}x_{i}^{-1}g^{-1})|\leq
|\psi(x_{i+1}x_{i}^{-1})|P(|\psi(x)|,|\psi(y)|,|\psi(g)|)$$ and then
$$k(gx_i,gx_{i+1})\leq
P(|\psi(x)|,|\psi(y)|,|\psi(g)|)|k(x_i,x_{i+1})|$$ We conclude,
\begin{align*}
d(gx,gy)&\leq k(gx_0,gx_1)+\ldots+k(gx_{n-1},gx_n) \\
&\leq P(|\psi(x)|,|\psi(y)|,|\psi(g)|)(k(x_0,x_1)+\ldots+k(x_{n-1},x_n))
\end{align*}
and the lemma follows taking the infimum.
\end{proof}

\begin{lem} \label{acotpol}
Let $g\in G$ and $n\in \N$, then $|\psi(g^n)|\leq P(n)$, where $P$ is a polynomial
with coefficients depending on $|\psi(g)|$.
\end{lem}

\begin{proof}
By the multiplication formula we observe that
$(\psi(g^n))_1=n\psi(g)_1$, i.e. the first coordinate is controlled
polynomially. Suppose now that the $i$-th coordinate is controlled
polynomially, then the same happens with the $i+1$-th coordinate. In
fact, we see inductively that,
$$\psi(g^{n+1})_{i+1}=\psi(g^ng)_{i+1}=\psi(g^n)_{i+1}+\psi(g)_{i+1}+P
(\psi(g^n)_1,\ldots,\psi(g^n)_i,\psi(g)_1,\ldots,\psi(g)_i)$$ and
then $|\psi(g^{n+1})_{i+1}|-|\psi(g^n)_{i+1}|\leq P(n),$ and we
conclude $|\psi(g^{n+1})_{i+1}|\leq (n+1)P(n+1)$. (Here all
polynomials $P$ are not necessarily the same.)
\end{proof}

\begin{lem}[Factorization]
Each $g\in G$ can be written in a unique way as $g=\{g\}[g]$ with
$\psi(\{g\})\in [0,1)^m$ and $[g]\in G$.
\end{lem}

Therefore, when writing $x\Gamma$ we can assume $x$ is such that $|\psi(x)|\leq 1$.

\begin{lem} \label{bound gamma}
Let $x,y \in G$. Then,
$$d(x\Gamma,y\Gamma)=\inf_{\gamma \in \Gamma} d(x,y\gamma)=\inf_{\gamma \in \Gamma,
|\psi(\gamma)|\leq C} d(x,y\gamma)$$
where $C$ is a constant depending only on $|\psi(x)|$ and $|\psi(y)|$.
\end{lem}

Combining the last two lemmas we see that:

$$d(x\Gamma,y\Gamma)=\inf_{\gamma \in \Gamma, |\psi(\gamma)|\leq C} d(x,y\gamma)$$
where $C$ is a constant.\\

We obtain,
$$d(g^nx\Gamma,g^ny\Gamma)\leq \inf_{\gamma \in \Gamma, |\psi(\gamma)|\leq C} d(g^nx,g^ny\gamma)$$

Using Lemmas \ref{polybound left}, \ref{acotpol} and \ref{bound gamma} we get,

\begin{cor}
 Let $x,y,g\in G$ with $|\psi(x)|\leq 1$ and $|\psi(y)|\leq 1$. Then,

\begin{equation}\label{eq:nildynamic}
d(g^nx\Gamma,g^ny\Gamma)\leq \inf_{\gamma \in \Gamma,
|\psi(\gamma)|\leq C}P(|\psi(g^n)|) \ d(x,y\gamma)\leq
P(n) \ d(x\Gamma,y\Gamma)
\end{equation}
\end{cor}

\subsection{Complexity computation}
Let $(X,T)$ be a t.d.s. We say that a subset $F\subseteq X$
is $(n,\epsilon)$-shadowing if for all $x \in X$ there exists $y \in F$
such that $d(T^ix,T^iy)\leq \epsilon$ for all $i\in \{0,\ldots,n\}$.
Write $r(n,\epsilon)=\min \{ |F|:F\subseteq X, F \textrm{ is }
(n,\epsilon)-\textrm{shadowing} \}$. Given an open cover
$\U=\{U_1,\ldots, U_k\}$ the topological complexity function of
$\mathcal{U}$ is the sequence on $n$: $c(\mathcal{U},n)=\min\{ M\geq 1:
\exists V_1,\ldots, V_M\in
\bigvee_{i=0}^{n}T^{-i}\mathcal{U},\textrm{ such that }
\bigcup\limits_{i=1}^M V_i=X\}$. One has that,
$$c(\mathcal{U},n)\leq r(n,\frac{\delta}{2})$$
where $\delta>0$ is the Lebesgue number of $\mathcal{U}$.

To compute the topological complexity of a nilsystem
in the general case (i.e. when $G$ is not neccesarily a connected and
simply connected Lie group) we will use an argument given by A.Leibman
in \cite{Le}. For that we require an extra definition and one lemma from
\cite{Le}.

Let $X=G/\Gamma$ be a nilmanifold and $T:X\to X$ the transformation
given by $x\to gx$ with a fixed $g\in G$.

\begin{defn}
We say that a closed subset $Y \subset X$ is a submanifold of $X$ if
$Y=Hx$ where $H$ is a closed subgroup of $G$ and $x \in X$.
\end{defn}

\begin{lem}\label{lem:reduction}
There exists a connected, simply connected nilpotent Lie group
$\widehat{G}$ and $\widehat{\Gamma}\subseteq \widehat{G}$ a co-compact
subgroup such that $X$ with the action of $G$ is isomorphic to a
submanifold $\widetilde{X}$ of
$\widehat{X}=\widehat{G}/\widehat{\Gamma}$ representing the action
of $G$ in $\widehat{G}$.
\end{lem}

This is the main result of the section.

\begin{thm}
Let $X=G/\Gamma$ be a nilmanifold and $T:X\to X$ the transformation
given by $x\to gx$ with a fixed $g\in G$. Let $\mathcal{U}$ be an open cover
of $X$. Then, for all $n\in \N$,
$$ c(\mathcal{U},n)\leq P(n)$$ where $P$ is a polynomial.
\end{thm}

\begin{proof}
First, assume that $G$ is a connected, simply connected nilpotent
Lie Group. For $\epsilon>0$, let $N(\epsilon)$ be the smallest
number of balls of ratio $\epsilon$ needed to cover $X$. The upper
Minkowski dimension or box dimension (see \cite{Po}) is defined
by
$$\limsup\limits_{\epsilon\to 0} \frac{\log N(\epsilon)}{\log(1/\epsilon)}$$
This dimension coincides with the usual dimension of the manifold
$X$ and hence there exists a constant $K$ such that:
$$N(\epsilon)\leq K (\frac{1}{\epsilon})^{\dim(X)+1}$$

Using the bound in \eqref{eq:nildynamic} we observe that if $x,y\in X$
and $d(x,y)\leq \frac{\delta}{2 P(n)}$, then $d(T^ix,T^iy)\leq
\frac{P(i)}{P(n)}\frac{\delta}{2}\leq \frac{\delta}{2}$ if $i\leq n$
since $P$ has positive coefficients. Let $\delta$ the Lebesgue
number of $\mathcal{U}$. We get,
$$c(\mathcal{U},n)\leq N\left (\frac{\delta}{ 2P(n)} \right )
\leq K\frac{(2P(n))^{\dim(X)+1}}{\delta^{\dim(X)+1}}$$
and the polynomial bound of the complexity is obtained.
\medskip

Now consider the general case. Denote by $\pi :X\to\widetilde{X}$
the isomorphism given by Lemma \ref{lem:reduction}. We see that
$(X,T)$ is conjugate with $(\widetilde{X},\widetilde{T})$ where
$\widetilde{T}:\widehat{X}\to \widehat{X}$ is defined by
$\widetilde{T}(\widehat{x})=\pi(g)\widehat{x}$. Hence
$(\widetilde{X},\widetilde{T})$ is a subsystem of
$(\widehat{X},\widetilde{T})$.
\medskip

If $\mathcal{U}=(U_1,\ldots,U_m)$ is an open cover of $X$,
$\pi(\mathcal{U})=(\pi(U_1),\ldots,\pi(U_m))$ is an open cover of
$\widetilde{X}$ and $c(\mathcal{U},n)=c(\pi(\mathcal{U}),n)\leq
c(\widehat{\mathcal{U}},n)$ where
$\widehat{\pi}(U)=(V_1,\ldots,V_m,\widetilde{X}^c)$ is an open cover
of $\widehat{X}$ with $\pi(U_i)=\widetilde{X}\cap V_i$ for all $i\in
\{1,\ldots,m\}$. Thus, we get a polynomial bound of the complexity in the
general case.
\end{proof}

Finally, we consider the complexity of a $\infty$-step nilsystem.
For that we need the following easy lemma.
\begin{lem}
Suppose $X$ is an inverse limit of the systems $(X_i,T)_{i\in \N}$
where  $(X_i,T)$ has a polynomial complexity for each $i \in \N$.
Then $X$ has polynomial complexity.
\end{lem}

\begin{proof}
We will show that the product system has polynomial complexity and
therefore the inverse limit has the same property. Let $\epsilon>0$ and choose
$N\in \N$ such that $\delta=\epsilon-2^{-N}>0$. Then
$r_{X}(\epsilon,n)\leq \prod\limits_{i\leq N} r_{X_i}(\delta,n)$ and
by assumption the right side is polynomially bounded.
\end{proof}

We conclude,

\begin{thm}
If $X$ is an $\infty$-step nilsystem then it has a polynomial
complexity.
\end{thm}

\begin{proof}
By the above discussion a $d$-step nilsystem has polynomial
complexity. By Theorem \cite{HKM} the factors $(X_d,T)$ defined by
the relation $\RP^{[d]}$ are inverse limits of $d$-step nilsystems
and therefore they have polynomial complexity. Using again the inverse limit argument
we conclude the polynomial bound for the complexity of $(X,T)$.
\end{proof}

\section{Appendix}

In this appendix we give the proof of Theorem \ref{noteasy}. First
we discuss Theorem \ref{noteasy} (2). The idea to prove this fact
was inspired from personal communications with B. Kra \cite{kra},
here we give details of the proof.

\begin{lem}\cite{Host}\label{h-k}
Let $(X,\mu,T)$ be an ergodic $d$-step nilsystem with $X=G/\Gamma$,
$\mu$ be its Haar probability measure and $T$ be the translation by
the element $t\in G$. Moreover, assume that the group $G$ can be
spanned by the connected component of the identity and the element
$t$ (it is always possible to reduce to this case, see \cite{BHK}).
Let $d\geq 1$ be an integer. If $Z_k$ is the maximal factor of order
$d$ of $(X,\mu,T)$ with $k\leq d$, then $Z_k$ has the form
$G/(G_{k+1}\Gamma)$ endowed with the translation by the projection
of $t$ on $G/G_{k+1}$, where $G_1=G, G_2=[G,G_1],
G_3=[G,G_2],\ldots, G_{d+1}=\{e\}$.
\end{lem}

\medskip

Now we prove Theorem \ref{noteasy} (2):

\medskip

\noindent{\it Proof of Theorem \ref{noteasy} (2):} Let $n>d$ be any
integer, we will show $Z_n=Z_d$. In \cite{HK05} it was shown that
$Z_n$ is an inverse limit of $n$-step nilsystems $(Z_{n,i})_{i\in
\N}$. For any $Z_{n,i}$, assume it has the form of $G/\Gamma$, where
the group $G$ is spanned by the connected component of the identity
and the translation element $t$.

Let $G^o$ be the identity component of $G$. Just as showed in
\cite{BHK}, $G_2=[G,G]= [G^o,G]$ is connected; and inductively for
any integer $k\geq 2$, $G_k$ is connected. By Lemma \ref{h-k}, the
$d$-step maximal nilfactor and $d+1$-step maximal nilfactor of
$G/\Gamma$ is $G/(G_{d+1}\Gamma)$ and $G/(G_{d+2}\Gamma)$
respectively. We have that $G/(G_{d+2}\Gamma)$ is also a $d+1$-step
nilfactor of $X$, so it is a factor of $Z_{d+1}=Z_d$, which implies
that $G/(G_{d+2}\Gamma)$ is also a $d$-step nilsystem. Now, by the
maximality of $G/(G_{d+1}\Gamma)$, we have $G/(G_{d+2}\Gamma)$ and
$G/(G_{d+1}\Gamma)$ coincide. Then, $G_{d+1}\Gamma=G_{d+2}\Gamma$,
and of course the nilpotent Lie groups $G_{d+1}$ and $G_{d+2}$ have
the same dimension since $\Gamma$ is discrete.

For any positive integer $k$, let $\mathfrak{g}_k$ be the associated
Lie algebra of $G_k$. Then $\mathfrak{g}_{d+1}$ and
$\mathfrak{g}_{d+2}$ have the same dimension, which implies that
$\mathfrak{g}_{d+1}$ and $\mathfrak{g}_{d+2}$ coincide since
$\mathfrak{g}_{d+2}$ is a subalgebra of $\mathfrak{g}_{d+1}$. Since
$G_{d+1}$ and $G_{d+2}$ are connected,
$G_{d+1}=\exp(\mathfrak{g}_{d+1})=\exp(\mathfrak{g}_{d+2})=G_{d+2}$,
and then we have that
$$G_{d+3}=[G,G_{d+2}]=[G,G_{d+1}]=G_{d+2}=G_{d+1}.$$
Inductively, we have $G_{k+1}=G_{d+1}$ for all $d\leq k\leq n$,
which implies that $G_{d+1}=\{e\}$, $G$ is $d$-step nilpotent and
$Z_{n,i}=G/\Gamma$ is a $d$-step nilsystem. So the inverse limit
$Z_n$ is a $d$-step nilsystem. By the maximality of $Z_d$ we
conclude $Z_n=Z_d$. The proof is completed. \hfill $\square$
\medskip

To show the next lemma we need some results from \cite{HKM}.

\begin{prop}\label{h-k-m} Let $(X,T)$ be a minimal system. Then,
If $(X,T)$ is an inverse limit of some $d$-step nilsystems,
then each measurable factor is a topological factor.

\end{prop}
\begin{proof} It is a combination of \cite[Proposition 5.2]{HKM}
and \cite[Lemma 6.1 ]{HKM}.
\end{proof}

\begin{lem}\label{basic}
Let $(X,T)$ be a minimal system of order $n$,
then the maximal measurable and topological factors of order $d$
coincide, where $d\leq n$.
\end{lem}
\begin{proof}
Let $\mu$ be the unique invariant probability measure of $X$ and
$Z_d$ is the maximal measurable factor of order $d$ of $(X,\mu,T)$.
It is clear that $Z_d$ is a topological factor of order $d$ of
$(X,\mu,T)$ by Proposition \ref{h-k-m}. Endow the maximal
topological factor $X_d$ of order $d$ of $X$ with its unique
invariant probability measure. Clearly it is a measurable factor of
order $d$ of $(X,\mu,T)$ and so is a measurable factor of $Z_d$. By
Proposition \ref{h-k-m} again, $X_d$ is a topological factor of
$Z_d$, and so $Z_d=X_d$ is the maximal topological factor of order
$d$.
\end{proof}
\medskip

Now we can finish the proof of Theorem \ref{noteasy}.

\medskip

\noindent{\it Proof of Theorem \ref{noteasy}:} (1) For any $k\geq1$,
recall $X_k=X/\RP^{[k]}$ and let $\mu_k$ be its unique invariant
probability measure. Let $n>d$ be an integer. By the Lemma
\ref{basic} and since $X_n$ is a minimal system of order $n$, then
the maximal measurable and topological factors of order $k$ of $X_n$
coincide and $X_n/\RP^{[k]}(X_n)=X_k=Z_k$, $k\le n$. As
$(Z_d,\mu_d,T)=(Z_{d+1},\mu_{d+1},T)$, by Theorem \ref{noteasy}(2),
we have for any $d\leq k\leq n$, $(Z_k,\mu_k,T)=(Z_d,\mu_d,T)$.
Therefore $(Z_n,\mu_n,T)$ and $(Z_d,\mu_d,T)$ coincide in the
measurable sense, and by Proposition \ref{h-k-m}, they coincide in
the topological sense too, i.e. $X_n=Z_n=Z_d=X_d$, which implies
that $\RP^{[n]}(X)=\RP^{[d]}(X)$.
\medskip

(3) If $Z_n$ is measure theoretical isomorphic with
$X_{n}=X/\RP^{[n]}(X)$ for some $n\in\N$. By the Lemma \ref{basic},
for any positive integer $k \le n$, the measurable and topological
maximal factors of order $k$ coincide, which implies $Z_k$ is
measure theoretical isomorphic with $X_{k}=X/\RP^{[k]}(X)$. The
proof of the theorem is completed. \hfill $\square$


\medskip









\begin{thebibliography}{SSS}

\bibitem{Au88} J. Auslander, {\em Minimal flows and their
extensions}, North-Holland Mathematics Studies, {\bfseries 153}, (1988),
North-Holland, Amsterdam.

\bibitem{Au} L. Auslander, L. Green and F. Hahn, {\em Flows on homogeneous spaces},
Annals of Mathematics Studies, {\bfseries 53}, Princeton University Press, Princeton, N.J. 1963 vii+107 pp.

\bibitem{BHK} V. Bergelson, B. Host and B. Kra, {\em
Multiple recurrence and nilsequences}, Invent. Math., {\bfseries 160},
(2005), 261--303, with an appendix by I.Z. Ruzsa.

\bibitem{BHM} F. Blanchard, B. Host and A. Maass, {\em
Topological complexity}, Ergodic Theory and Dynamical Systems, {\bfseries 20},
(2000), 641--662.

\bibitem{BL} F. Blanchard and Y. Lacroix, {\em Zero-entropy factors of
topological flows}, Proc. Amer. Math. Soc., {\bfseries 119}, (1993),
985--992.

\bibitem{F61} H. Furstenberg, {\em Strict ergodicity and transformation of the torus},
Amer. J. Math., {\bfseries 83}, (1961), 573--601.

\bibitem{F} H. Furstenberg, {\em Recurrence in ergodic theory and
combinatorial number theory}, M. B. Porter Lectures, Princeton
University Press, Princeton, N.J., 1981.

\bibitem{G76}
E. Glasner, {\em Proximal flows\/}, Lecture Notes in Math.\
{\bfseries 517}, Springer-Verlag, 1976.

\bibitem{G05} E. Glasner, {\em Topological weak mixing and quasi-Bohr systems},
Israel J. Math., {\bfseries 148}, (2005), 277--304.

\bibitem{G07} E. Glasner, {\em The structure of tame minimal dynamical systems},
Ergodic Theory and Dynamical Systems, {\bfseries 27}, (2007), 1819--1837.

\bibitem{GW} E. Glasner and B. Weiss, {\em Quasi-factors of zero-entropy systems},
J. Amer. Math. Soc., {\bfseries 8}, (1995), 665--686.

\bibitem{GY} E. Glasner and X. Ye, {\em Local entropy theory}, Ergodic
Theory and Dynamical Systems, {\bfseries 29}, (2009), 321--356.

\bibitem{GT} B. Green and T. Tao, {\em The quantitative behaviour of polynomial
orbits on nilmanifolds}, to appear in Annals of Math.

\bibitem{Host} B. Host, {\em Convergence of multiple ergodic averages}, arXiv:0606362.

\bibitem{HK05} B. Host and B. Kra, {\em Nonconventional averages and
nilmanifolds}, Ann. of Math., {\bfseries 161}, (2005) 398--488.

\bibitem{HKM} B. Host, B. Kra and A. Maass, {\em Nilsequences and a Structure
Theory for Topological Dynamical Systems}, Advances in Mathematics,
{\bfseries 224}, (2010) 103--129.

\bibitem{H1} W. Huang, {\em Tame systems and scrambled pairs under an Abelian group action},
Ergodic Theory and Dynamical Systems, {\bfseries 26},
(2006), 1549--1567.

\bibitem{HLY} W. Huang, H. Li and X. Ye, {\em Family-independence for topological and measurable dynamics},
arXiv: 0908.0574, to appear in Trans. Amer. Math. Soc..

\bibitem{HLY1} W. Huang, H. Li and X. Ye, {\em Localization and dynamical Ramsey
property}, Preprint.

\bibitem{HLSY} W. Huang, S. Li, S. Shao and X. Ye, {\em Null systems and sequence entropy pairs},
Ergodic Theory and Dynamical Systems, {\bfseries 23}, (2003), 1505-1523.

\bibitem {HY1} W. Huang  and X. Ye, {\em Topological
complexity, return times and weak disjointness}, Ergodic
Theory and Dynamical Systems, {\bfseries 24}, (2004), 825--846.

\bibitem{HY} W. Huang and X. Ye, {\em A local variational
relation and applications}, Israel J. Math., {\bfseries 151}, (2006),
237--280.

\bibitem{kra} B. Kra, Personal communication.


\bibitem{KL} D. Kerr and H. Li, {\em Independence in topological and
$C^*$-dynamics}, Math. Ann., {\bfseries 338}, (2007), 869--926.

\bibitem{Le} A. Leibman, {\em Pointwise convergence of ergodic averages for polynomial
sequences of translations on a nilmanifold}, Ergodic Theory and Dynamical Systems, {\bfseries 25}, (2005), no. 1, 201--213.

\bibitem{L} E. Lindenstrauss, {\em Measurable distal and topological distal
systems}, Ergodic Theory and Dynamical Systems, {\bfseries 19}, (1999), no.4,
1063--1076.

\bibitem{MS} A. Maass and S. Shao, {\em Sequence entropy in minimal
systems}, J. London of Math. Soc., {\bfseries 76}, (2007), no. 3, 702--718.


\bibitem{Po} M. Pollicott, {\it Fractals and Dimension Theory}.
Available at http://www.warwick.ac.uk/~masdbl/preprints.html

\bibitem{SY} S. Shao and X. Ye, \textit{Regionally proximal relation of order $d$ is
an equivalence one for minimal systems and a combinatorial
consequence}, arXiv:1007.0189.

\bibitem{V70}
W. A. Veech, {\em Point-distal flows\/}, Amer. J. Math., {\bfseries 92}, (1970),  205--242.

\bibitem{V77}
W. A. Veech, {\em Topological dynamics}, Bull.\ Amer.\ Math. Soc.,
{\bfseries 83}, (1977), 775--830.

\bibitem{We} B. Weiss, {\em Multiple recurrence and doubly minimal systems},
Contemporary Math., {\bf 215}, (1998), 189--196.

\bibitem{W} J. van der Woude, \textit{Topological dynamics}, Dissertation,
Vrije Universiteit, Amsterdam, 1982. CWI Tract, {\bf 22}.

\bibitem{Z} T. Ziegler, \textit{Universal characteristic factors and Furstenberg
averages}. J. Amer. Math. Soc. 20 (2007), no. 1, 53--97.

\bibitem{An} A. Zygmund, \textit{Trigonometric Series}, second edition, University Press, Cambridge, 1959.

\end{thebibliography}
\end{document}